\documentclass{amsart}

\usepackage{amsthm,amsmath,amsfonts,amssymb,mathabx}
\usepackage{xcolor}
\usepackage{tikz}

\newtheorem{theorem}{Theorem}[subsection]
\newtheorem{lemma}[theorem]{Lemma}
\newtheorem{proposition}[theorem]{Proposition}
\newtheorem{corollary}[theorem]{Corollary}

\theoremstyle{definition}
\newtheorem{definition}[theorem]{Definition}

\theoremstyle{remark}

\numberwithin{equation}{section}

\newcommand{\des}{\ensuremath\mathrm{des}}
\newcommand{\Des}{\ensuremath\mathrm{Des}}

\newcommand{\weakarrow}{\ensuremath\leftsquigarrow}
\newcommand{\wP}{\ensuremath\widetilde{P}}
\newcommand{\wQ}{\ensuremath\widetilde{Q}}

\newcommand{\wD}{\ensuremath\widetilde{d}}

\newcommand{\word}{\ensuremath\mathrm{word}}

\newcommand{\SYT}{\ensuremath\mathrm{SYT}}

\newcommand{\row}{\ensuremath\mathrm{row}}

\newcommand{\R}{\ensuremath{R}}
\newcommand{\Yam}{\ensuremath{YR}}

\newcommand{\D}{\ensuremath\mathbb{D}}

\newcommand{\lift}{\ensuremath\mathrm{lift}}
\newcommand{\drop}{\ensuremath\mathrm{drop}}

\newcommand{\SKT}{\ensuremath\mathrm{SKT}}

\newcommand{\swap}{\ensuremath\mathfrak{c}}
\newcommand{\braid}{\ensuremath\mathfrak{b}}

\newcommand{\CK}{\ensuremath\mathfrak{d}}

\newcommand{\flatten}{\ensuremath\mathrm{flat}}
\newcommand{\schubert}{\ensuremath\mathfrak{S}}
\newcommand{\fund}{\ensuremath\mathfrak{F}}
\newcommand{\key}{\ensuremath\kappa}

\newcommand{\stanley}{\ensuremath\mathrm{S}}

\newlength\cellsize 

\savebox2{%
\begin{picture}(12,12)
\put(0,0){\line(1,0){12}}
\put(0,0){\line(0,1){12}}
\put(12,0){\line(0,1){12}}
\put(0,12){\line(1,0){12}}
\end{picture}}

\newcommand\cellify[1]{\def\thearg{#1}\def\nothing{}%
\ifx\thearg\nothing\vrule width0pt height\cellsize depth0pt%
  \else\hbox to 0pt{\usebox2\hss}\fi%
  \vbox to 12\unitlength{\vss\hbox to 12\unitlength{\hss$#1$\hss}\vss}}

\newcommand\tableau[1]{\vtop{\let\\=\cr
\setlength\baselineskip{-12000pt}
\setlength\lineskiplimit{12000pt}
\setlength\lineskip{0pt}
\halign{&\cellify{##}\cr#1\crcr}}}

\newcommand\nocellify[1]{\def\thearg{#1}\def\nothing{}%
\ifx\thearg\nothing\vrule width0pt height\cellsize depth0pt%
  \else\hbox to 0pt{\hss}\fi%
  \vbox to 12\unitlength{\vss\hbox to 12\unitlength{\hss$#1$\hss}\vss}}

\newcommand\notableau[1]{\vtop{\let\\=\cr
\setlength\baselineskip{-12000pt}
\setlength\lineskiplimit{12000pt}
\setlength\lineskip{0pt}
\halign{&\nocellify{##}\cr#1\crcr}}}

\savebox4{
\begin{picture}(10,10)
\put(5,5){\circle{10}}
\end{picture}}

\newcommand{\cir}[1]{\def\thearg{#1}\def\nothing{}%
\ifx\thearg\nothing\vrule width0pt height\cellsize depth0pt%
  \else\hbox to 0pt{\usebox4\hss}\fi%
  \vbox to 10\unitlength{\vss\hbox to 10\unitlength{\hss$#1$\hss}\vss}}

\usepackage[colorinlistoftodos]{todonotes}

\begin{document}


\title[Generalized Edelman--Greene insertion]{A generalization of Edelman--Greene insertion \\ for Schubert polynomials}  

\author{Sami Assaf}
\address{Department of Mathematics, University of Southern California, 3620 South Vermont Avenue, Los Angeles, CA 90089-2532, U.S.A.}
\email{shassaf@usc.edu}
\thanks{Work supported in part by NSF DMS-1763336.}






\begin{abstract}
  Edelman and Greene generalized the Robinson--Schensted--Knuth correspondence to reduced words in order to give a bijective proof of the Schur positivity of Stanley symmetric functions. Stanley symmetric functions may be regarded as the stable limits of Schubert polynomials, and similarly Schur functions may be regarded as the stable limits of Demazure characters for the general linear group. We modify the Edelman--Greene correspondence to give an analogous, explicit formula for the Demazure character expansion of Schubert polynomials. Our techniques utilize dual equivalence and its polynomial variation, but here we demonstrate how to extract explicit formulas from that machinery which may be applied to other positivity problems as well.
\end{abstract}

\maketitle

%
\section{Introduction}
%
\label{sec:introduction}

Schur functions, the ubiquitous basis for symmetric functions with deep connections to representation theory and geometry, may be regarded as the generating functions for standard Young tableaux. In an analogous way, Stanley \cite{Sta84} defined a generating function for reduced words that he proved was symmetric and conjectured was Schur positive. Edelman and Greene \cite{EG87} established a bijective correspondence between reduced words and ordered pairs of Young tableaux of the same partition shape such that the left is increasing with reduced reading word and the right is standard. Thus through this correspondence they proved Stanley's conjecture and, moreover, gave an explicit formula for the Schur expansion as the number of such left tableaux that can appear in the correspondence. 

Schubert polynomials were introduced by Lascoux and Sch{\"u}tzenberger \cite{LS82} as polynomial representatives of Schubert classes for the cohomology of the flag manifold with nice algebraic and combinatorial properties. They can be defined as the generating polynomials of reduced words \cite{BJS93,AS17}, and in the stable limit, they become the Stanley symmetric functions \cite{Mac91}. Parallel to this, Demazure characters for the general linear group \cite{Dem74a} can be regarded as the generating polynomials for standard key tableaux \cite{AS18,Ass-W}, and in the stable limit, the key tableaux become Young tableaux and the Demazure characters become Schur functions \cite{LS90}. Lascoux and Sch{\"u}tzenberger \cite{LS90} noticed that the Schubert polynomials expand nonnegatively into Demazure characters parallel to the nonnegative expansion of Stanley symmetric functions into Schur functions. The proof \cite{LS90,RS95} uses the same structure of partitioning reduced words into equivalence classes based on the Edelman--Greene right tableau, yet there is no direct formula for the coefficents given. 

In this paper, we complete the analogy between the function and polynomial settings by providing a new bijective correspondence between reduced words and ordered pairs of key tableaux of the same weak composition shape such that the left is \emph{Yamanouchi} with reduced reading word and the right is standard. Thus through this correspondence we prove the Demazure positivity of Schubert polynomials and, moreover, give an explicit formula for the Demazure expansion as the number of such left tableaux that can appear in the correspondence.  

Our real purpose, in addition to this explicit result, is to provide a framework by which one can extract explicit Schur expansions of symmetric functions through the machinery of \emph{dual equivalence} and explicit Demazure expansions of polynomials through the machinery of \emph{weak dual equivalence}. Dual equivalence and its weak variant give universal methods for proving positivity results, but they do so indirectly without giving tractable formulas. Our techniques in this paper utilize dual equivalence to show how it can manipulated to give the desired formulas.

This paper is structured as follows. We begin in Section~\ref{sec:genfun} with a review of definitions for reduced words for permutations. We develop parallel theories of the generating functions for reduced words, reviewing Stanley symmetric functions \cite{Sta84} in Section~\ref{sec:genfun-stanley} and generating polynomials for reduced words, reviewing Schubert polynomials \cite{LS82} in Section~\ref{sec:genfun-schubert}.

Questions of positivity arise in Section~\ref{sec:equiv}, where we consider the Coxeter--Knuth equivalence relations \cite{EG87} on reduced words. Maintaining our parallel study, in Section~\ref{sec:CK-dual}, we review the machinery of dual equivalence \cite{Ass15} to see that the generating function of a Coxeter--Knuth equivalence class on reduced words is a Schur function, thus recovering the Schur positivity result of Edelman and Greene for Stanley symmetric functions \cite{EG87}. In Section~\ref{sec:CK-weak}, we review the machinery of weak dual equivalence \cite{Ass-W} to see that the generating polynomial of a Coxeter--Knuth equivalence class on reduced words is a Demazure character, thus recovering the Demazure positivity result of Lascoux and Sch{\"u}tzenberger for Schubert polynomials \cite{LS90,RS95}. 

In Section~\ref{sec:positive}, we embark on the quest to extract explicit formulas for these expansions by finding canonical representatives for the Coxeter--Knuth equivalence classes. In Section~\ref{sec:positive-increasing}, we recover the explicit formula of Edelman and Greene for the Schur expansion Stanley symmetric functions \cite{EG87} using simple techniques that avoid the subtlety of their insertion algorithm. In Section~\ref{sec:positive-yamanouchi} we use similar techiques to arrive at our main result: an explicit combinatorial formula for the Demazure expansion of Schubert polynomials. 

Finally, we return to the inspiration of this work in Section~\ref{sec:insertion}, where we present explicit insertion algorithms. In Section~\ref{sec:insertion-EG}, we review the Edelman--Greene correspondence that associates to each reduced word a pair Young tableaux, and then in Section~\ref{sec:insertion-weak} we use results from Section~\ref{sec:positive} to modify this correspondence to associate to each reduced word a pair of key tableaux. In this way, we complete the parallel stories with satisfactory formulas for both cases.

%

%
\section{Generating functions for reduced words}
%
\label{sec:genfun}

The symmetric group $\mathcal{S}_n$ has generators $s_i$, the simple transpositions interchanging $i$ and $i+1$, and relations $s_i^2$ is the identity, $s_i s_j = s_j s_i$ for $|i-j| \geq 2$, and $s_i s_{i+1} s_i = s_{i+1} s_i s_{i+1}$ for $1\leq i \leq n-2$.

An \emph{expression} for a permutation $w \in \mathcal{S}_n$ is a way of writing $w$ in terms of these simple generators, i.e. $w = s_{\rho_{\ell}} \cdots s_{\rho_1}$. The \emph{length} of $w$, denoted by $\ell(w)$ is the number of pairs $(i<j)$ such that $w_i > w_j$. If an expression for $w$ has exactly $\ell(w)$ terms, then it is \emph{reduced}. In this case, the sequence of indices $\rho = (\rho_{\ell(w)}, \ldots, \rho_1)$ such that $w = s_{\rho_{\ell(w)}} \cdots s_{\rho_1}$ is called a \emph{reduced word} for $w$. 

For example, there are two reduced expressions for the permutation $321$, namely $s_1 s_2 s_1$ and $s_2 s_1 s_2$, both of which have length $3$ since there are $3$ inversions in $321$. Therefore we say that $(1,2,1)$ and $(2,1,2)$ are reduced words for $321$. For a more elaborate example, Fig.~\ref{fig:reduced} shows the $11$ reduced words for the permutation $153264$.

\begin{figure}[ht]
  \begin{displaymath}
    \arraycolsep=5pt
    \begin{array}{cccccc}
      (5 , 3 , 2,3,4) & (5 , 2,3 , 2,4) & (5 , 2,3,4 , 2) & (3,5 , 2,3,4) &
      (3 , 2,5 , 3,4) & (3 , 2,3,5 , 4) \\ (2,5 , 3,4 , 2) & (2,3,5 , 4 , 2) &
      (2,5 , 3 , 2,4) & (2,3,5 , 2,4) & (2,3 , 2,5 , 4) & 
    \end{array}
  \end{displaymath}
   \caption{\label{fig:reduced}The set of reduced words for $w = 153264$.}
\end{figure}

We consider the set $\R(w)$ of reduced words for a given permutation $w$. Below we present two different generating functions for this set by assigning either a symmetric function or a polynomial to each reduced word.

\subsection{Stanley symmetric functions}
\label{sec:genfun-stanley}

Stanley \cite{Sta84} defined a family of symmetric functions in order to enumerate reduced words. These functions have since been realized to have important connections with geometry and representation theory.

A \emph{weak composition} $a = (a_1,\ldots,a_{n})$ is a sequence of nonnegative integers. A \emph{composition} $\alpha = (\alpha_1,\ldots,\alpha_{\ell})$ is a sequence of positive integers.  A \emph{partition} $\lambda = (\lambda_1 \geq \cdots \geq \lambda_{\ell})$ is a weakly decreasing sequence of positive integers. Given compositions $\alpha,\beta$, we say \emph{$\beta$ refines $\alpha$} if there exist indices $i_1<\ldots<i_{\ell}$ such that
\begin{displaymath}
  \beta_1 + \cdots + \beta_{i_j} = \alpha_1 + \cdots + \alpha_j.
\end{displaymath}
For example, $(1,2,2)$ refines $(3,2)$ but does not refine $(2,3)$.

Gessel introduced the \emph{fundamental quasisymmetric functions} \cite{Ges84}, indexed by compositions, that form an important basis for quasisymmetric functions.

\begin{definition}[\cite{Ges84}]
  For $\alpha$ a composition, the \emph{fundamental quasisymmetric function} $F_{\alpha}$ is 
  \begin{equation}
    F_{\alpha}(x_1,x_2,\ldots) = \sum_{\flatten(b) \ \mathrm{refines} \ \alpha} x_1^{b_1} x_2^{b_2} \cdots,
    \label{e:F_n}
  \end{equation}
  where the sum is over weak compositions $b$ for which the composition $\flatten(b)$ obtained by removing all parts equal to $0$ refines $\alpha$.
\end{definition}
  
For example, restricting to three variables to make the expansion finite, we have
\begin{displaymath}
  F_{(3,2)}(x_1,x_2,x_3) = x_2^3 x_3^2 + x_1^3 x_3^2 + x_1^3 x_2^2 + x_1^3 x_2 x_3 + x_1 x_2^2 x_3^2 + x_1^2 x_2 x_3^2 .
\end{displaymath}

Stanley \cite{Sta84} defined a family of symmetric functions indexed by permutations that are the fundamental quasisymmetric generating functions for reduced words. To define this, we associate a composition to each reduced word.

\begin{definition}
  The \emph{run decomposition} of a reduced word $\rho$ partitions $\rho$ into increasing sequences $\rho=(\rho^{(k)} | \cdots | \rho^{(1)})$ of maximal length. The \emph{descent composition of $\rho$}, denoted by $\Des(\rho)$, is the composition $(|\rho^{(1)}|,\ldots,|\rho^{(k)}|)$. 
  \label{def:Des}
\end{definition}

For example, $\rho = (3,6,4,7,5,2,4)$ and $\sigma = (6, 7, 3, 4, 5, 2, 4)$, two reduced words for $w = 13625847$, have run decompositions
\[ (\overbrace{3,6}^{\rho^{(4)}} \mid \overbrace{3,4}^{\rho^{(3)}} \mid \overbrace{5}^{\rho^{(2)}} \mid \overbrace{2,4}^{\rho^{(1)}}) \hspace{2em} \text{and} \hspace{2em} (\overbrace{6,7}^{\sigma^{(3)}} \mid \overbrace{3,4,5}^{\sigma^{(2)}} \mid \overbrace{2,4}^{\sigma^{(1)}}), \]
giving $\Des(\rho) = (2,1,2,2)$ and $\Des(\sigma) = (2,3,2)$. Note the reversal of lengths. 

We may visualize the descent composition via positive integer fillings of cell diagrams as follows.

\begin{definition}
  The \emph{descent tableau} of a reduced word $\rho$, denoted by $\D(\rho)$, is the filling of unit cells in the first quadrant constructed as follows. Place $\rho_{\ell(w)}$ into the first column of row $|\Des(\rho)|$. For $i = \ell(w)-1,\ldots,2,1$, place $\rho_i$ immediately right of $\rho_{i+1}$ if $\rho_{i+1} < \rho_i$; otherwise in the first column of the next row down.
  \label{def:d-red}
\end{definition}

For example, $\rho = (3,6,4,7,5,2,4)$ is inserted as shown on the left side of Figure~\ref{fig:Des}, and $\sigma = (6, 7, 3, 4, 5, 2, 4)$ inserts as shown on the right side of Figure~\ref{fig:Des}. 

\begin{figure}[ht]
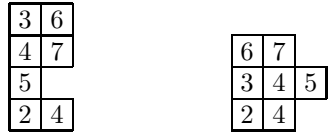

  \begin{displaymath}
    \tableau{3 & 6 \\ 4 & 7 \\ 5 \\ 2 & 4 }
    \hspace{5\cellsize}
    \tableau{\\ 6 & 7 \\ 3 & 4 & 5 \\ 2 & 4 }
  \end{displaymath}
   \caption{\label{fig:Des}Constructing the descent tableaux for reduced words $\rho = (3,6,4,7,5,2,4)$ (left) and $\sigma=(6, 7, 3, 4, 5, 2, 4)$ (right).}
\end{figure}

By construction, rows of $\D(\rho)$ are increasing and the descent composition for $\rho$ is given by the lengths of the rows of the descent tableau for $\rho$, read bottom to top. 

\begin{definition}[\cite{Sta84}]
  For $w$ a permutation, the \emph{Stanley symmetric function} $\stanley_{w}$ is 
  \begin{equation}
    \stanley_{w}(X) = \sum_{\rho \in \R(w)} F_{\Des(\rho)}(X).
    \label{e:stanley}
  \end{equation}
  \label{def:stanley}
\end{definition}

For example, from Fig.~\ref{fig:reduced}, we compute
\begin{eqnarray*}
  \stanley_{153264}(X) & = & F_{(3,1,1)}(X) + 2 F_{(2,2,1)}(X) + 2 F_{(1,3,1)}(X) + F_{(3,2)}(X) \\
  & & + 2 F_{(1,2,2)}(X) + F_{(1,1,3)}(X) + F_{(2,1,2)}(X) + F_{(2,3)}(X) .
\end{eqnarray*}

\subsection{Schubert polynomials}
\label{sec:genfun-schubert}

Lascoux and Sch{\"u}tzenberger \cite{LS82} introduced \emph{Schubert polynomials} as a basis for the polynomial ring that gives polynomial representatives of Schubert classes for the cohomology of the flag manifold with nice algebraic and combinatorial properties. 

Assaf and Searles \cite{AS17} introduced \emph{fundamental slide generating polynomials} as a generalization of the fundamental quasisymmetric functions that form a basis for the full polynomial ring. They showed that Schubert polynomials are the fundamental slide generating polynomials for reduced words.

\begin{definition}[\cite{AS17}]
  For a weak composition $a$ of length $n$, the \emph{fundamental slide polynomial} $\fund_{a} = \fund_{a}(x_1,\ldots,x_n)$ is
  \begin{equation}
    \fund_{a} = \sum_{\substack{b_1 + \cdots + b_k \geq a_1 + \cdots + a_k \ \forall k \\ \mathrm{flat}(b) \ \mathrm{refines} \ \mathrm{flat}(a)}} x_1^{b_1} \cdots x_n^{b_n},
    \label{e:fund-shift}
  \end{equation}
  where $\mathrm{flat}(a)$ denotes the composition obtained by removing all zero parts.
  \label{def:fund-shift}
\end{definition}

For example, we have
\begin{displaymath}
  \fund_{(3,0,2)}(x_1,x_2,x_3) = x_1^3 x_3^2 + x_1^3 x_2^2 + x_1^3 x_2 x_3 ,
\end{displaymath}
which is not equal to $F_{(3,2)}(x_1,x_2,x_3)$ computed earlier. However, we do have $\fund_{(0,3,2)} = F_{(3,2)}(x_1,x_2,x_3)$. Moreover, Assaf and Searles proved fundamental slide polynomials stabilize \cite{AS17}(Theorem~4.5).

\begin{proposition}[\cite{AS17}]
  For a weak composition $a$, we have
  \begin{equation}
    \lim_{m \rightarrow \infty} \fund_{0^m \times a}(x_1,\ldots,x_m,0,\ldots,0) = F_{\flatten(a)}(X),
  \end{equation}
  where $0^m \times a$ is the weak composition obtained by prepending $m$ $0$'s to $a$.
  \label{prop:fund-stable}
\end{proposition}

We generalize the descent composition of a reduced word to a weak composition as defined in \cite{Ass-T}(Definition~3.2).

\begin{definition}[\cite{Ass-T}]
  For a reduced word $\rho$, define the \emph{weak descent composition of $\rho$}, denoted by $\des(\rho)$, as follows. Let $(\rho^{(k)} | \cdots | \rho^{(1)})$ be the run decomposition of $\rho$, that is, each $\rho^{(i)}$ is increasing and as long as possible. Set $r_i = \min(\rho^{(i)})$ for $i=1,\ldots,k$. Set $\hat{r}_k = r_k$, and for $i<k$, set $\hat{r}_i = \min(r_i,\hat{r}_{i+1}-1)$. If $r_1 \leq 0$, then define $\des(\rho) = \varnothing$; otherwise, set the $\hat{r}_i$th part of $\des(\rho)$ to be $\des(\rho)_{\hat{r}_i} = |\rho^{(i)}|$ and set all other parts to $0$.
  \label{def:des-red}
\end{definition}

Revisiting our previous examples, for $\rho = (3,6 \mid 4,7 \mid 5 \mid 2,4)$, we have $r_4 = 3$, $\hat{r}_3 = \min(4,3-1) = 2$, $\hat{r}_2 = \min(5,2-1) = 1$, and finally $\hat{r}_1 = \min(2,1-1) = 0$ giving $\des(\rho) = \varnothing$ since $\hat{r}_1 =0$. For $\sigma = (6, 7 \mid 3, 4, 5 \mid 2, 4)$, we have $s_3 = 6$, $\hat{s}_2 = \min(3,6-1) = 3$, and $\hat{s}_2 = \min(2,3-1) = 2$, giving $\des(\sigma) = (0,2,3,0,0,2)$.

We may visualize Definition~\ref{def:des-red} via a simple insertion algorithm as follows.

\begin{definition}
  The \emph{weak descent tableau} of a reduced word $\rho$, denoted by $\D(\rho)$, is the following filling of unit cells in the right half plane. Place $\rho_{\ell(w)}$ into the first column of row $\rho_{\ell(w)}$. For $i = \ell(w)-1,\ldots,2,1$, place $\rho_i$ immediately right of $\rho_{i+1}$ if $\rho_{i+1} < \rho_i$, or in the first column of the lower of row $\rho_{i}$ or the row below $\rho_{i+1}$.
  \label{def:diag-red}
\end{definition}

For example, $\rho = (3,6,4,7,5,2,4)$ is inserted as shown on the left side of Figure~\ref{fig:des}, and $\sigma = (6, 7, 3, 4, 5, 2, 4)$ inserts as shown on the right side of Figure~\ref{fig:des}.

\begin{figure}[ht]
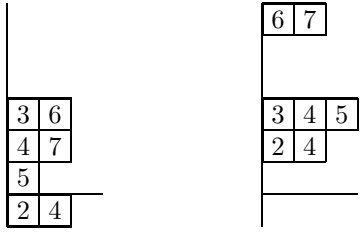

  \begin{displaymath}
    \vline\tableau{ \\ \\ \\ 3 & 6 \\ 4 & 7 & \\ 5 \\\hline 2 & 4 }
    \hspace{5\cellsize}
    \vline\tableau{6 & 7 \\ \\ \\ 3 & 4 & 5 \\ 2 & 4 \\  \\\hline & }
  \end{displaymath}
   \caption{\label{fig:des}Constructing the weak descent tableaux for reduced words $\rho = (3,6,4,7,5,2,4)$ (left) and $\sigma=(6, 7, 3, 4, 5, 2, 4)$ (right).}
\end{figure}

Notice $\des(\rho) = \varnothing$ if and only if there is an occupied row of $\D(\rho)$ with non-positive index, and otherwise $\des(\rho)_i$ is the number of entries in row $i$ of $\D(\rho)$.

We say that $\rho$ is \emph{virtual} if $\des(\rho) = \varnothing$. To facilitate virtual objects, set
\begin{equation}
  \fund_{\varnothing} = 0.
\end{equation}

Billey, Jockusch, and Stanley \cite{BJS93}(Theorem~1.1) gave a combinatorial definition for the monomial expansion of Schubert polynomials in terms of compatible sequences for reduced words. Assaf and Searles \cite{AS17}(Theorem~3.13) refined this to give a combinatorial model for the expansion in terms of fundamental slide polynomials. The re-formulation of the latter given below appears in \cite{Ass-T}(Theorem~3.3), and we take this as our definition.

\begin{definition}
  For $w$ any permutation, the \emph{Schubert polynomial} $\schubert_w$ is 
  \begin{equation}
    \schubert_{w} = \sum_{\rho \in \R(w)} \fund_{\des(P)},
    \label{e:schubert-slide}
  \end{equation}
  where the sum may be taken over non-virtual reduced words $\rho$.
  \label{def:schubert}
\end{definition}

For example, the seven non-virtual reduced words for $\R(153264)$ give
\begin{eqnarray*}
  \schubert_{153264} & = & \fund_{(0,3,1,0,1)} + \fund_{(2,2,0,0,1)} + \fund_{(1,3,0,0,1)} + \fund_{(0,3,2,0,0)} \\ & & + \fund_{(2,2,1,0,0)} + \fund_{(1,3,1,0,0)} + \fund_{(2,3,0,0,0)}.
\end{eqnarray*}

Macdonald \cite{Mac91}(7.18) showed that Schubert polynomials stabilize and that their stable limits are precisely the Stanley symmetric functions. This follows directly from Proposition~\ref{prop:fund-stable} by Definitions~\ref{def:schubert} and \ref{def:stanley} as well.

\begin{proposition}[\cite{Mac91}]
  For $w$ a permutation, we have
  \begin{equation}
    \lim_{m \rightarrow \infty} \schubert_{1^m \times w}(x_1,\ldots,x_m,0,\ldots,0) = \stanley_{w}(X),
  \end{equation}
  where $1^m \times w$ is the permutation obtained by adding $m$ to each $w_i$ and then prepending $1 2 \cdots m$.
  \label{prop:stanley-stable}
\end{proposition}

%
\section{Equivalence relations}
%
\label{sec:equiv}

We consider simple involutions based on the Coxeter relations for the simple transpositions that generate the symmetric group.

Given $\rho\in\R(w)$, for $1\leq j<\ell(w)$, let $\swap_j$ denote the \emph{commutation relation} that acts by exchanging $\rho_j$ and $\rho_{j+1}$ if $|\rho_j-\rho_{j+1}|>1$ and the identity otherwise.

Given $\rho\in\R(w)$, for $1 < j<\ell(w)$, let $\braid_j$ denote the \emph{braid relation} that acts by sending $\rho_{j+1} \rho_{j} \rho_{j-1}$ to $\rho_{j} \rho_{j+1} \rho_{j}$ if $\rho_{j+1} = \rho_{j-1}$ and the identity otherwise.

Any reduced words in the same equivalence class under $\{\swap_j,\braid_j\}$ are called \emph{Coxeter equivalent}. A classical result of Tits \cite{Tit69} states that each set $\R(w)$ is a single Coxeter equivalence class. For examples, see Fig.~\ref{fig:coxeter}.

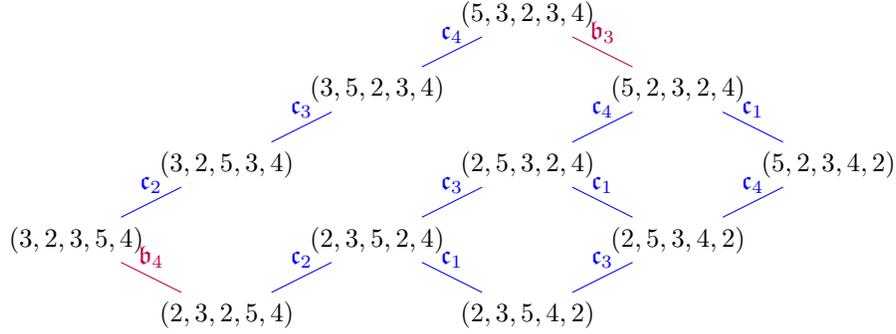
\begin{figure}[ht]
  \begin{center}
    \begin{tikzpicture}[xscale=2,yscale=1]
      \node at (3,5) (C5) {$(5,3,2,3,4)$};
      \node at (2,4) (B4) {$(3,5,2,3,4)$};
      \node at (4,4) (D4) {$(5,2,3,2,4)$};
      \node at (1,3) (A3) {$(3,2,5,3,4)$};
      \node at (3,3) (C3) {$(2,5,3,2,4)$};
      \node at (5,3) (E3) {$(5,2,3,4,2)$};
      \node at (0,2) (A2) {$(3,2,3,5,4)$};
      \node at (2,2) (C2) {$(2,3,5,2,4)$};
      \node at (4,2) (E2) {$(2,5,3,4,2)$};
      \node at (1,1) (B1) {$(2,3,2,5,4)$};
      \node at (3,1) (D1) {$(2,3,5,4,2)$};
      \draw[thin,color=blue  ] (C5) -- (B4) node[midway,above] {$\swap_4 $};
      \draw[thin,color=purple] (C5) -- (D4) node[midway,above] {$\braid_3$};
      \draw[thin,color=blue  ] (B4) -- (A3) node[midway,above] {$\swap_3 $};
      \draw[thin,color=blue  ] (D4) -- (C3) node[midway,above] {$\swap_4 $};
      \draw[thin,color=blue  ] (D4) -- (E3) node[midway,above] {$\swap_1 $};
      \draw[thin,color=blue  ] (A3) -- (A2) node[midway,above] {$\swap_2 $};
      \draw[thin,color=blue  ] (C3) -- (C2) node[midway,above] {$\swap_3 $};
      \draw[thin,color=blue  ] (C3) -- (E2) node[midway,above] {$\swap_1 $};
      \draw[thin,color=blue  ] (E3) -- (E2) node[midway,above] {$\swap_4 $};
      \draw[thin,color=purple] (A2) -- (B1) node[midway,above] {$\braid_4$};
      \draw[thin,color=blue  ] (C2) -- (B1) node[midway,above] {$\swap_2 $};
      \draw[thin,color=blue  ] (C2) -- (D1) node[midway,above] {$\swap_1 $};
      \draw[thin,color=blue  ] (E2) -- (D1) node[midway,above] {$\swap_3 $};
    \end{tikzpicture}
  \caption{\label{fig:coxeter}An illustration of the Coxeter relation involutions on $\R(153264)$.}
  \end{center}
\end{figure}

Knuth \cite{Knu70} considered relations on permutations that characterize when two permutations give rise to the same Schensted insertion tableau \cite{Sch61}. Analogously, Edelman and Greene \cite{EG87} characterize when two reduced words give rise to the same Edelman--Greene insertion tableau using \emph{elementary Coxeter-Knuth relations}.

\begin{definition}
  For $1<i<\ell(w)$, the \emph{elementary Coxeter-Knuth relation} $\CK_i$ acts on a reduced word $\rho \in\R(w)$ by
  \begin{equation}
    \CK_i (\rho) = \left\{ \begin{array}{rl}
      \braid_i   (\rho) & \mbox{if } \rho_{i+1} = \rho_{i-1} (= \rho_{i} \pm 1) \\
      \swap_{i-1}(\rho) & \mbox{if } \rho_{i-1} > \rho_{i+1} > \rho_{i} \mbox{ or } \rho_{i-1} < \rho_{i+1} < \rho_{i}, \\
      \swap_{i}  (\rho) & \mbox{if } \rho_{i+1} > \rho_{i-1} > \rho_{i}  \mbox{ or } \rho_{i+1} < \rho_{i-1} < \rho_{i}, \\
      \rho & \mbox{otherwise},
    \end{array} \right.
  \end{equation}
  where $\swap_j$ denotes a \emph{commutation relation} and $\braid_j$ denotes a \emph{braid relation}. 
  \label{def:dual-stanley}
\end{definition}

We partition $\R(w)$ by stating any reduced words in the same equivalence class under $\{\CK_i\}$ are  \emph{Coxeter-Knuth equivalent}. For example, see Fig.~\ref{fig:CK-dual}.

\begin{figure}[ht]
  \begin{center}
    \begin{tikzpicture}[xscale=2,yscale=1]
      \node at (3,5) (C5) {$(5,3,2,3,4)$};
      \node at (2,4) (B4) {$(3,5,2,3,4)$};
      \node at (4,4) (D4) {$(5,2,3,2,4)$};
      \node at (1,3) (A3) {$(3,2,5,3,4)$};
      \node at (3,3) (C3) {$(2,5,3,2,4)$};
      \node at (5,3) (E3) {$(5,2,3,4,2)$};
      \node at (0,2) (A2) {$(3,2,3,5,4)$};
      \node at (2,2) (C2) {$(2,3,5,2,4)$};
      \node at (4,2) (E2) {$(2,5,3,4,2)$};
      \node at (1,1) (B1) {$(2,3,2,5,4)$};
      \node at (3,1) (D1) {$(2,3,5,4,2)$};
      \draw[thin,color=purple] (C5) -- (D4) node[midway,above] {$\CK_3$};
      \draw[thin,color=purple] (B4) -- (A3) node[midway,below] {$\CK_3$};
      \draw[thin,color=violet] (B4) -- (A3) node[midway,above] {$\CK_4$};
      \draw[thin,color=violet] (D4) -- (C3) node[midway,above] {$\CK_4$};
      \draw[thin,color=blue  ] (D4) -- (E3) node[midway,above] {$\CK_2$};
      \draw[thin,color=blue  ] (A3) -- (A2) node[midway,above] {$\CK_2$};
      \draw[thin,color=blue  ] (C3) -- (E2) node[midway,above] {$\CK_2$};
      \draw[thin,color=violet] (E3) -- (E2) node[midway,above] {$\CK_4$};
      \draw[thin,color=violet] (A2) -- (B1) node[midway,above] {$\CK_4$};
      \draw[thin,color=purple] (C2) -- (B1) node[midway,below] {$\CK_3$};
      \draw[thin,color=blue  ] (C2) -- (B1) node[midway,above] {$\CK_2$};
      \draw[thin,color=purple] (E2) -- (D1) node[midway,above] {$\CK_3 $};
    \end{tikzpicture}
  \caption{\label{fig:CK-dual}The partitioning of $\R(153264)$ into two Coxeter--Knuth equivalence classes.}
  \end{center}
\end{figure}
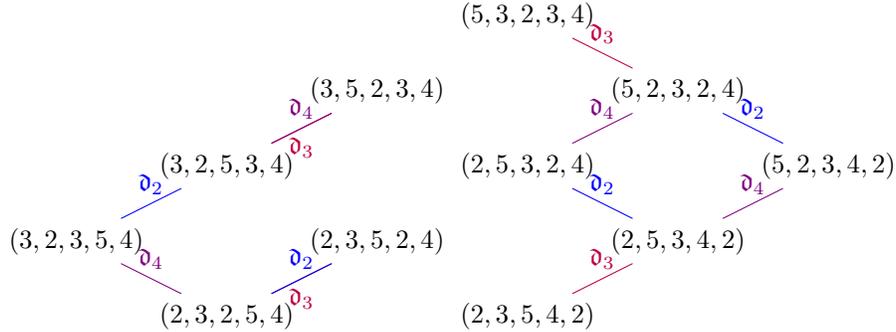

Inverting history, a natural question to ask is whether this partitioning can be realized on the level of symmetric functions by decomposing the Stanley symmetric functions or on the level of polynomials by decomposing Schubert polynomials.

\subsection{Dual equivalence}
\label{sec:CK-dual}

Based on the explicit \emph{elementary dual equivalence involutions} on standard Young tableaux, Assaf \cite{Ass07,Ass15} defined an abstract notion of dual equivalence that can be used to prove that a given fundamental quasisymmetric generating function is symmetric and \emph{Schur positive}.

A \emph{Young diagram} is the set of unit cells in the first quadrant with $\lambda_i$ cells in row $i$ for some partition $\lambda$. A \emph{Young tableau} is a filling of a Young diagram with positive integers. A Young tableau is \emph{increasing} if it has strictly increasing rows (left to right) and columns (bottom to top). A Young tableau is \emph{standard} if it is increasing and uses each integer $1,2,\ldots,n$ exactly once. For example, Figure~\ref{fig:SYT} shows the standard Young tableaux of shape $(3,2)$. 

\begin{figure}[ht]
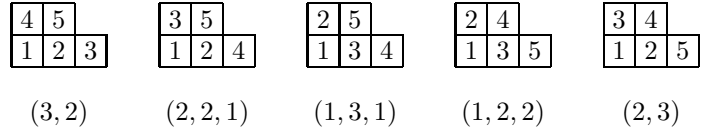

  \begin{displaymath}
    \begin{array}{c@{\hskip 2em}c@{\hskip 2em}c@{\hskip 2em}c@{\hskip 2em}c}
      \tableau{4 & 5 \\ 1 & 2 & 3} & 
      \tableau{3 & 5 \\ 1 & 2 & 4} & 
      \tableau{2 & 5 \\ 1 & 3 & 4} &
      \tableau{2 & 4 \\ 1 & 3 & 5} & 
      \tableau{3 & 4 \\ 1 & 2 & 5} \\ \\
      (3,2) &
      (2,2,1) &
      (1,3,1) &
      (1,2,2) &
      (2,3) 
    \end{array}
  \end{displaymath}
  \caption{\label{fig:SYT}The standard Young tableaux for $\lambda = (3,2)$ and their descent compositions.}
\end{figure}

For a standard Young tableau $T$, say $i$ is a \emph{descent of $T$} if $i+1$ lies weakly left of $i$. The \emph{descent composition of $T$}, denoted by $\Des(T)$, is the strong composition given by maximal lengths of runs of the word $12\cdots n$ crossing no descents. For example, the descent compositions for the tableaux in Figure~\ref{fig:SYT}.

Schur functions may be defined combinatorially as the fundamental quasisymmetric generating functions for standard Young tableaux. This follows from the classical definition (see \cite{Mac95}) by results of Gessel \cite{Ges84}.

\begin{definition}
  For $\lambda$ a partition, the \emph{Schur function} $s_{\lambda}$ is 
  \begin{equation}
    s_{\lambda}(X) = \sum_{T \in \SYT(\lambda)} F_{\Des(T)}(X).
    \label{e:schur-F}
  \end{equation}
  \label{def:schur-F}
\end{definition}

For example, from Figure~\ref{fig:SYT} we have
\begin{displaymath}
  s_{(3,2)}(X) = F_{(2,3)}(X) + F_{(1,2,2)}(X) + F_{(1,3,1)}(X) + F_{(3,2)}(X) + F_{(2,2,1)}(X) .
\end{displaymath}

A \emph{dual equivalence} for a set of objects endowed with a descent statistic is a family of involutions $d_2,\ldots,d_{n-1}$ such that $d_i d_j = d_j d_i$ for $|i-j|\geq 3$ and for which each restricted equivalence classes under $d_i,\ldots,d_j$ for $j-i \leq 3$ is a single Schur function. The main theorem for dual equivalence \cite{Ass15}(Corollary~4.4) states that this local Schur positivity implies global Schur positivity.

\begin{theorem}[\cite{Ass15}]
  The fundamental quasisymmetric generating function of any dual equivalence class is a single Schur function.
\end{theorem}

Stanley proved that $\stanley_{w}$ is symmetric \cite{Sta84}(Theorem~2.1) and conjectured that it is, in fact, \emph{Schur positive}, meaning the expansion into Schur functions has only nonnegative coefficients. For example,
\begin{displaymath}
  \stanley_{153264}(X) = s_{(3,2)}(X) + s_{(3,1,1)}(X).
\end{displaymath}

Edelman and Greene \cite{EG87} proved this by generalizing the Robinson--Schensted--Knuth insertion algorithm \cite{Rob38,Sch61,Knu70} on permutations.

\begin{theorem}[\cite{EG87}]
  For $w$ a permutation, we have
  \begin{equation}
    \stanley_{w} = \sum_{\substack{\rho\in\R(w) \\ \exists T_{\rho} \ \text{increasing}, \ \row(T_{\rho})=\rho}} s_{\Des(\rho)} ,
    \label{e:stanley-schur}
  \end{equation}
  where $\row(T)$ is the row reading word (left to right along rows from the top) of $T$. 
  \label{thm:EG}
\end{theorem}

For example, from Fig.~\ref{fig:coxeter}, the two reduced words for $153264$ that are the row reading words of increasing tableaux are $(3,5 \mid 2,3,4)$ and $(5 \mid 3 \mid 2,3,4)$, corresponding to the Schur expansion of $\stanley_{153264}$ given above.

The Edelman--Greene correspondence is an elegant solution to the Schur positivity conjecture, but the arguments involved in the proof require intricate analysis of bumping paths with many separate cases. Thus one can hope to find a simpler proof that avoids much of this subtlety.

Edelman and Greene \cite{EG87}(Corollary~6.15) relate Coxeter--Knuth equivalence with dual equivalence through the Edelman--Greene recording tableaux. Implicit in their work and explicit in \cite{Ass-W}(Theorem~2.10), the Coxeter--Knuth involutions give a dual equivalence on reduced words.

\begin{theorem}[\cite{Ass-W}]
  The Coxeter--Knuth involutions $\{\CK_i\}$ give a dual equivalence for $\R(w)$, thus Stanley symmetric functions are symmetric and Schur positive.
  \label{thm:deg-red}
\end{theorem}

That is, the Coxeter--Knuth relations $\CK_i$ partition reduced words for a given permutation into dual equivalence classes, each of which has fundamental quasisymmetric generating function equal to a single Schur function. However, while the proof of Theorem~\ref{thm:deg-red} is simple, the resulting formula requires computing each equivalence class in its entirety, falling short of the explicit formula in Theorem~\ref{thm:EG}.

\subsection{Weak dual equivalence}
\label{sec:CK-weak}

The \emph{Demazure characters}, introduced by Demazure \cite{Dem74}, originally arose as characters of Demazure modules for the general linear group \cite{Dem74a}. These polynomials were studied combinatorially by Lascoux and Sch{\"u}tzenberger \cite{LS90}, who call them \emph{standard bases}, and more extensively by Reiner and Shimozono \cite{RS95} who call them \emph{key polynomials}. We use the key tableaux model \cite{Ass-W} based on ideas of Kohnert \cite{Koh91} developed further by Assaf and Searles \cite{AS18}.

A \emph{key diagram} is a collection of left-justified unit cells in the right half place with $a_i$ cells in row $i$ for some weak composition $a$. A \emph{key tableau} is a filling of a key diagram with positive integers. The definition for \emph{standard key tableaux} \cite{Ass-W}(Definition~3.10) is more subtle than for standard Young tableaux.

\begin{definition}[\cite{Ass-W}]
  A \emph{standard key tableau} is a bijective filling of a key diagram with $\{1,2,\ldots,n\}$ such that rows decrease (left to right) and if some entry $i$ is above and in the same column as an entry $k$ with $i<k$, then there is an entry right of $k$, say $j$, such that $i<j$. 
  \label{def:key-tab}
\end{definition}

\begin{figure}[ht]
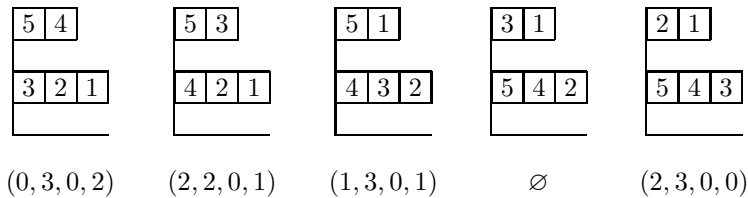

  \begin{displaymath}
    \begin{array}{c@{\hskip 2em}c@{\hskip 2em}c@{\hskip 2em}c@{\hskip 2em}c}
      \vline\tableau{5 & 4 \\ \\ 3 & 2 & 1 \\ \\\hline} &
      \vline\tableau{5 & 3 \\ \\ 4 & 2 & 1 \\ \\\hline} &
      \vline\tableau{5 & 1 \\ \\ 4 & 3 & 2 \\ \\\hline} &
      \vline\tableau{3 & 1 \\ \\ 5 & 4 & 2 \\ \\\hline} &
      \vline\tableau{2 & 1 \\ \\ 5 & 4 & 3 \\ \\\hline} \\ \\
      (0,3,0,2) & (2,2,0,1) & (1,3,0,1) & \varnothing & (2,3,0,0) 
    \end{array}
  \end{displaymath}
  \caption{\label{fig:SKT}Standard key tableaux of shape $(0,3,0,2)$ and their weak descent compositions.}
\end{figure}

For a standard key tableau $T$, say $i$ is a \emph{descent of $T$} if $i+1$ lies weakly right of $i$ in $T$. Note that this is the reverse of the concept of descents for standard Young tableaux. Next we define a \emph{weak descent composition} \cite{Ass-W}(Definition~3.12). 

\begin{definition}[\cite{Ass-W}]
  For a standard key tableau $T$, define the \emph{weak descent composition of $T$}, denoted by $\des(T)$, as follows. Let $(\tau^{(k)} | \cdots | \tau^{(1)})$ be the run decomposition of $n \cdots 2 1$ based on descents of $T$, that is, each $\tau^{(i)}$ has no descents between adjacent letters and is as long as possible. Set $t_i = \min(\tau^{(i)})$ for $i=1,\ldots,k$. Set $\hat{t}_k = r_k$, and for $i<k$, set $\hat{t}_i = \min(r_i,\hat{t}_{i+1}-1)$. If $t_1 \leq 0$, then define $\des(\rho) = \varnothing$; otherwise, set the $\hat{t}_i$th part of $\des(T)$ to be $\des(T)_{\hat{t}_i} = |\tau^{(i)}|$ and set all other parts to $0$.
  \label{def:des-key}
\end{definition}

For example, the standard key tableaux of shape $(0,3,0,2)$ shown in Figure~\ref{fig:SKT} have weak descent compositions given beneath.

Using this notion, we have the following reformulation of Demazure characters given in \cite{Ass-W}(Corollary~3.16) that we take as our definition.

\begin{definition}
  Given a weak composition $a$, we have
  \begin{equation}
    \key_a = \sum_{T \in \SKT(a)} \fund_{\des(T)} .
    \label{e:key-fund}
  \end{equation}
  \label{def:key}
\end{definition}

For example, from Figure~\ref{fig:SKT} we compute
\begin{displaymath}
  \key_{(0,3,0,2)} = \fund_{(0,3,0,2)} + \fund_{(2,2,0,1)} + \fund_{(1,3,0,1)} + \fund_{(2,3,0,0)}.
\end{displaymath}

Implicit in the work of Lascoux and Sch{\"u}tzenberger \cite{LS90} and explicit in that of Assaf and Searles \cite{AS18}(Corollary~4.9), we have the following analog of Proposition~\ref{prop:stanley-stable} for Demazure characters.

\begin{proposition}[\cite{AS18}]
  For a weak composition $a$, the \emph{Demazure character} is
  \begin{equation}
    \lim_{m \rightarrow \infty} \key_{0^m \times a}(x_1,\ldots,x_m,0,\ldots,0) = s_{\mathrm{sort}(a)}(X),
  \end{equation}
  where $\mathrm{sort}(a)$ is the partition obtained by sorting $a$ into weakly decreasing order.
  \label{prop:key-stable}
\end{proposition}

Generalizing dual equivalence, a \emph{weak dual equivalence} is a family of involutions $\wD_2,\ldots,\wD_{n-1}$ that give a dual equivalence when weak descent compositions are flattened to descent compositions by removing parts equal to $0$ and for which the fundamental slide generating polynomial of each restricted equivalence class under $\wD_i,\ldots,\wD_j$ for $j-i \leq 3$ is a single Demazure character. Parallel to the symmetric case, the main theorem for weak dual equivalence \cite{Ass-W}(Theorem~3.33) states that this local Demazure positivity implies global Demazure positivity.

\begin{theorem}[\cite{Ass-W}]
  The fundamental slide generating polynomial of any weak dual equivalence class is a single Demazure character.
\end{theorem}

One might now anticipate that Schubert polynomials expand nonnegatively into Demazure characters, parallel to \eqref{e:stanley-schur}, and indeed, we have, 
\begin{displaymath}
  \schubert_{153264} = \key_{(0,3,1,0,1)} + \key_{(0,3,2,0,0)} .
\end{displaymath}

Lascoux and Sch{\"u}tzenberger \cite{LS90} give a formula for the key polynomial expansion of a Schubert polynomial as a sum over increasing Young tableau whose row reading word is a reduced word for $w$, where for each such $\rho$ one computes the \emph{left nil key} by considering all reduced words of $w$ that are Coxeter--Knuth equivalent to $\rho$. For details that fill the gaps in \cite{LS90}, see \cite{RS95}(Theorem~4). 

\begin{theorem}[\cite{LS90,RS95}]
  For $w$ a permutation, we have
  \begin{equation}
    \schubert_{w} = \sum_{\substack{\rho\in\R(w) \\ \exists T_{\rho} \ \text{increasing}, \ \row(T_{\rho})=\rho}} \key_{\mathrm{content}(K^0_{-}(\rho))},
    \label{e:nilkey}
  \end{equation}
  where $K^0_{-}(\rho)$ is the left nil key of $\rho$. 
  \label{thm:LS}
\end{theorem}

While theoretically interesting for the nonnegativity, this result does not provide a direct formula as one is required to compute each Coxeter--Knuth class, and so the computation is effectively equivalent to computing the fundamental slide expansion. A simplified proof comes as an immediate application of weak dual equivalence.

\begin{theorem}[\cite{Ass-W}]
  The Coxeter--Knuth involutions $\{\CK_i\}$ give a weak dual equivalence for $\R(w)$. In particular, Schubert polynomials are Demazure positive.
  \label{thm:wdeg-red}
\end{theorem}

That is, the Coxeter--Knuth relations $\CK_i$ partition reduced words for a given permutation into \emph{weak dual equivalence classes}, each of which has fundamental slide generating polynomial equal to a single Demazure character. This gives a simplified proof of Theorem~\ref{thm:LS}, though the resulting formula is no more tractable.

%
\section{Positive expansions}
%
\label{sec:positive}

By Theorem~\ref{thm:deg-red}, each Coxeter--Knuth equivalence class corresponds to a term in the Schur expansion of a Stanley symmetric function. Similarly, by Theorem~\ref{thm:wdeg-red}, each Coxeter--Knuth equivalence class corresponds to a term in the Demazure expansion of a Schubert polynomial. To make these positivity results more compelling, we wish to have canonical representatives from each Coxeter--Knuth equivalence class from which an exact formula can be easily computed.

Edelman and Greene \cite{EG87} resolved this for the Schur expansion of Stanley symmetric functions, but we wish to give a simple, self-contained proof of their formula that avoids the subtleties of their insertion algorithms. The end result will be the same, however, namely that each Coxeter--Knuth equivalence class contains a unique reduced word whose descent tableau is an increasing Young tableau. Then the shape of these tableaux determines the Schur expansion.

\begin{figure}[ht]
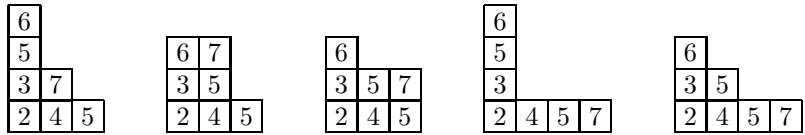

  \begin{displaymath}
    \begin{array}{c@{\hskip 2\cellsize}c@{\hskip 2\cellsize}c@{\hskip 2\cellsize}c@{\hskip 2\cellsize}c}
      \tableau{6 \\ 5 \\ 3 & 7 \\ 2 & 4 & 5} &
      \tableau{\\ 6 & 7 \\ 3 & 5 \\ 2 & 4 & 5} &
      \tableau{\\ 6 \\ 3 & 5 & 7 \\ 2 & 4 & 5} &
      \tableau{6 \\ 5 \\ 3 \\ 2 & 4 & 5 & 7} &
      \tableau{\\ 6 \\ 3 & 5 \\ 2 & 4 & 5 & 7} 
    \end{array}
  \end{displaymath}
   \caption{\label{fig:increasing}The set of increasing Young tableaux whose row reading words are reduced words for $w = 13625847$.}
\end{figure}

For example, we compute the Schur expansion of Stanley symmetric function $\stanley_{13625847}$ by constructing the five increasing Young tableaux in Figure~\ref{fig:increasing}, giving
\[ \stanley_{13625847} = s_{(3,2,1,1)} + s_{(3,2,2)} + s_{(3,3,1)} + s_{(4,1,1,1)} + s_{(4,2,1)} . \]

However, in the Schubert case, these are not the correct Coxeter--Knuth equivalence class representatives for giving the Demazure expansion. In light of Proposition~\ref{prop:key-stable}, there are many different candidates for which weak composition should index each class, even knowing the correct partition. Using the same techniques with which we prove the Edelman--Greene formula below, we also give an explicit algorithm to construct the correct Coxeter--Knuth equivalence class representatives for the polynomial case.

\begin{figure}[ht]
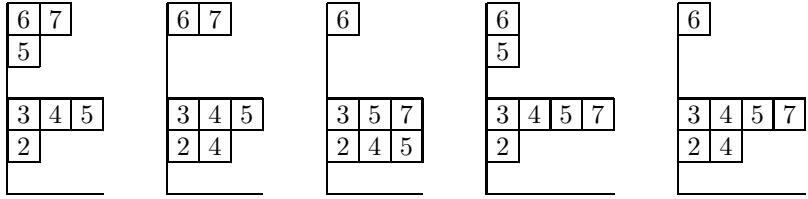

  \begin{displaymath}
    \begin{array}{c@{\hskip 2\cellsize}c@{\hskip 2\cellsize}c@{\hskip 2\cellsize}c@{\hskip 2\cellsize}c}
      \vline\tableau{6 & 7 \\ 5 \\ \\ 3 & 4 & 5 \\ 2 \\ \\\hline} &
      \vline\tableau{6 & 7 \\ \\ \\ 3 & 4 & 5 \\ 2 & 4 \\ \\\hline} &
      \vline\tableau{6 \\ \\ \\ 3 & 5 & 7 \\ 2 & 4 & 5 \\ \\\hline} &
      \vline\tableau{6 \\ 5 \\ \\ 3 & 4 & 5 & 7 \\ 2 \\ \\\hline} &
      \vline\tableau{6 \\ \\ \\ 3 & 4 & 5 & 7 \\ 2 & 4 \\ \\\hline} 
    \end{array}
  \end{displaymath}
   \caption{\label{fig:yam}The set of Yamanouchi key tableaux whose row reading words are reduced words for $w = 13625847$.}
\end{figure}

For example, we compute the Demazure expansion of the Schubert polynomial $\schubert_{13625847}$ by constructing the five \emph{Yamanouchi} key tableaux in Figure~\ref{fig:yam}, giving
\[ \schubert_{13625847} = \key_{(0,1,3,0,1,2)} + \key_{(0,2,3,0,0,2)} + \key_{(0,3,3,0,0,1)} + \key_{(0,1,4,0,1,1)} + \key_{(0,2,4,0,0,1)} . \]

\subsection{Increasing Young tableaux}
\label{sec:positive-increasing}

We begin by considering the descent tableaux for reduced words, and, more generally, any tableau with weakly increasing rows for which the reading word is reduced.

\begin{definition}
  Given two increasing words $\sigma,\tau$ of lengths $s,t$, respectively, define the \emph{drop alignment of $\sigma$ below $\tau$} as follows: If $\tau_j>\sigma_j$ for all $j\leq\min(s,t)$, then left justify $\sigma$ with respect to $\tau$. Otherwise, set $j_1$ to be the minimum index such that $\tau_j \leq \sigma_j$, align $\sigma_1\cdots\sigma_{j_1-1}$ directly under $\tau_1\cdots\tau_{j_1-1}$ and iterate the process with the drop alignment of $\sigma_{j_1}\cdots\sigma_s$ below $\tau_{j_1+1}\cdots\tau_{t}$.
  \label{def:drop-align}
\end{definition}

Visually, begin with $\sigma$ left justified under $\tau$, and from left to right, for each not strict column, slide entries of $\sigma$ from that column onward right by one position. For example, Fig.~\ref{fig:drop-ex} shows the drop alignments for two pairs of increasing words.

\begin{figure}[ht]
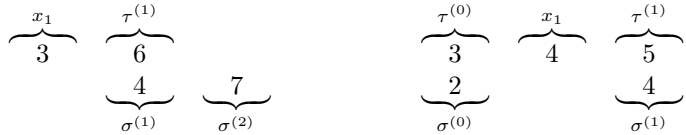

  \begin{displaymath}
    \begin{array}{cc}
      \begin{array}{ccc}
        \overbrace{3}^{x_1} & \overbrace{6}^{\tau^{(1)}} & \\ & \underbrace{4}_{\sigma^{(1)}} & \underbrace{7}_{\sigma^{(2)}}
      \end{array} & \hspace{3\cellsize}
      \begin{array}{ccc}
        \overbrace{3}^{\tau^{(0)}} & \overbrace{4}^{x_1} & \overbrace{5}^{\tau^{(1)}} \\ \underbrace{2}_{\sigma^{(0)}} & & \underbrace{4}_{\sigma^{(1)}}
      \end{array}
    \end{array} 
  \end{displaymath}
  \caption{\label{fig:drop-ex}The drop alignment of $(4,7)$ below $(3,6)$ (left) and of $(2,4)$ below $(3,4,5)$ (right).}
\end{figure}

If there are $k$ instances in the drop alignment of $\sigma$ below $\tau$ where a cell of $\tau$ has no cell below it, then we denote these cells of $\tau$ as $x_1,\ldots,x_k$ and factor $\tau=\tau^{(0)}x_1\tau^{(1)}\cdots x_k\tau^{(k)}$ and, correspondingly, $\sigma=\sigma^{(0)}\sigma^{(1)}\cdots\sigma^{(k)}\sigma^{(k+1)}$ as shown in Fig.~\ref{fig:drop_factor}. When the concatentation $\tau\sigma$ is a reduced word, then this factorization has the following properties.

\begin{figure}[ht]
  \begin{displaymath}
    \arraycolsep=1.5pt
    \begin{array}{lllllll}
      \tau^{(0)} & x_1 & \tau^{(1)} & \cdots & x_k & \tau^{(k)} & \\
      \sigma^{(0)} & & \sigma^{(1)} & \cdots & & \sigma^{(k)} & \sigma^{(k+1)} 
    \end{array} 
  \end{displaymath}
  \caption{\label{fig:drop_factor}An illustration of the drop alignment of $\sigma$ below $\tau$ and corresponding factorizations of $\tau$ and $\sigma$.}
\end{figure}

\begin{proposition}
  Given two increasing words $\tau,\sigma$ such that $\tau\sigma$ is reduced, there is a unique factorization $\tau=\tau^{(0)}x_1\tau^{(1)}\cdots x_k\tau^{(k)}$ and $\sigma=\sigma^{(0)}\sigma^{(1)}\cdots\sigma^{(k)}\sigma^{(k+1)}$, with some $\tau^{(i)}$ or $\sigma^{(i)}$ possibly empty, such that
  \begin{enumerate}
  \item $\ell(\tau^{(j)}) = \ell(\sigma^{(j)})$ for $j=1,\ldots,k$;
  \item $\tau^{(j)}_i > \sigma^{(j)}_i$ for $j=1,\ldots,k$ and $i=1,\ldots,\ell(\tau^{(j)})$;
  \item $x_j \leq \sigma^{(j)}_1$ with equality only if $\tau^{(j)}_1 = \sigma^{(j)}_1+1$, for $j=1,\ldots,k$.
  \end{enumerate}  
  \label{prop:drop}
\end{proposition}

\begin{proof}
  The algorithm in Definition~\ref{def:drop-align} is clearly well-defined and conditions (1) and (2) follow immediately by construction. The hypothesis that $\tau\sigma$ is reduced is needed only for condition (3). Suppose $j_1$ is the minimum index such that $\tau_j \leq \sigma_j$ and that, in fact, $\tau_{j_1} = \sigma_{j_1}$. Condition (3) requires $\tau_{j_1+1} = \tau_{j_1}+1$. Suppose, for contradiction, this does not hold. Then either $j_1=t$ or $\tau_{j_1+1} > \tau_{j_1}+1$.

  By minimality of $j_1$, we have $\tau_i>\sigma_i$ for all $i<j_1$. In particular, $\sigma_1,\ldots,\sigma_{j_1-1} < \tau_{j_1-1} < \tau_{j_1}$, so the word $\tau\sigma$ is Coxeter equivalent to the word
  \[ \tau_1 \cdots \tau_{j_1-1} \sigma_1 \cdots \sigma_{j_1-1} \tau_{j_1} \cdots \tau_t \sigma_{j_1} \cdots \sigma_s . \]
  If $j_1=t$, then $\tau_{j_1} = \sigma_{j_1}$ are adjacent, contradicting the fact that the word is reduced. If $j_1<t$, then $\tau_{j_1+1} > \tau_{j_1}+1 = \sigma_{j_1}+1$, and so the above word is Coxeter equivalent to the word
  \[ \tau_1 \cdots \tau_{j_1-1} \sigma_1 \cdots \sigma_{j_1-1} \tau_{j_1} \sigma_{j_1} \tau_{j_1+1}\cdots \tau_t \sigma_{j_1} \cdots \sigma_s . \]
  In this case as well, $\tau_{j_1} = \sigma_{j_1}$ are adjacent, contradicting the fact that the word is reduced. Thus condition (3) must hold whenever $\tau\sigma$ is reduced.  
\end{proof}

We call this factorization the \emph{drop} alignment because, as we show below, we may drop the unsupported cells $x_1,\ldots,x_k$ from $\tau$ down to $\sigma$ without changing the Coxeter--Knuth equivalence class.

\begin{definition}
  Given two increasing words $\tau,\sigma$ such that $\tau\sigma$ is reduced, define 
  \[ \drop(\tau\sigma) = \tau^{(0)} \tau^{(1)} \cdots \tau^{(k)} \sigma^{(0)} x_1 \hat{\sigma}^{(1)}\cdots x_k \hat{\sigma}^{(k)}\sigma^{(k+1)}, \]
  where $\tau=\tau^{(0)}x_1\tau^{(1)}\cdots x_k\tau^{(k)}$ and $\sigma=\sigma^{(0)}\sigma^{(1)}\cdots\sigma^{(k)}\sigma^{(k+1)}$ is the unique factorization of Proposition~\ref{prop:drop} and for $1 \leq j \leq k$, we set
  \[ \hat{\sigma}^{(j)}_i = \left\{ \begin{array}{ll}
    \sigma^{(j)}_i+1 & \text{for} \ 1 \leq i \leq b_j \\
    \sigma^{(j)}_i & \text{for} \ b_j+1 \leq i \leq \ell(\sigma^{(j)})
  \end{array} \right. \]
  where $b_j = \max\{b \mid  \tau^{(j)}_1 = \sigma^{(j)}_i+i \ \forall 1 \leq i \leq b \}$ if $x_j = \tau^{(j)}_1-1$ and $0$ otherwise.
  \label{def:drop-two}
\end{definition}

For example, dropping the aligned words in Fig.~\ref{fig:drop-ex} results in the words in Fig.~\ref{fig:drop-ex-2}.

\begin{figure}[ht]
  \begin{displaymath}
    \begin{array}{cc}
      \begin{array}{ccc}
        & \overbrace{6}^{\tau^{(1)}} & \\ \underbrace{3}_{x_1} & \underbrace{4}_{\sigma^{(1)}} & \underbrace{7}_{\sigma^{(2)}}
      \end{array} & \hspace{3\cellsize}
      \begin{array}{ccc}
        \overbrace{3}^{\tau^{(0)}} & & \overbrace{5}^{\tau^{(1)}} \\ \underbrace{2}_{\sigma^{(0)}} & \underbrace{4}_{x_1} & \underbrace{5}_{\hat{\sigma}^{(1)}}
      \end{array}
    \end{array} 
  \end{displaymath}
  \caption{\label{fig:drop-ex-2}The drop of $(4,7)$ below $(3,6)$ (left) and of $(2,4)$ below $(3,4,5)$ (right).}
\end{figure}

The following elementary lemma will be useful in proving that $\tau\sigma$ is Coxeter--Knuth equivalent to $\drop(\tau\sigma)$.

\begin{lemma}
  Given an increasing word $\sigma=(\sigma_{\ell},\ldots,\sigma_1)$ and a letter $x$ such that $x < \sigma_{\ell}$, we have $\sigma_{\ell}x\sigma_{\ell-1}\cdots\sigma_1$ is Coxeter--Knuth equivalent to $\sigma_{\ell}\sigma_{\ell-1}\cdots\sigma_1 x$.
  \label{lem:slide}
\end{lemma}

\begin{proof}
  We claim $\CK_{2}\cdots\CK_{\ell-1}(\sigma_{\ell}x\sigma_{\ell-1}\cdots\sigma_1) = \sigma x$, from which the assertion follows. To see this, notice that from the hypotheses we have $x < \sigma_{k} < \sigma_{k-1}$ for $k=\ell,\ldots,2$. Thus we may apply $\CK_{k-1}$, for $k=\ell,\ldots,3$, with effect
  \[ \CK_{k-1}(\sigma_{\ell}\cdots\sigma_{k}x\sigma_{k-1}\cdots\sigma_1) = \sigma_{\ell}\cdots\sigma_{k-1}x\sigma_{k-2}\cdots\sigma_1. \]
  The culmination of these elementary equivalences gives the desired result.
\end{proof}

\begin{lemma}[Drop Lemma]
  Given two increasing words $\tau,\sigma$ such that $\tau\sigma$ is reduced, $\drop(\tau\sigma)$ is Coxeter--Knuth equivalent to $\tau\sigma$.
  \label{lem:drop}
\end{lemma}

\begin{proof}
  With the factorization as denoted in Proposition~\ref{prop:drop}, consider first the case $k=1$, in which we must show $\tau^{(0)} x \tau^{(1)} \sigma^{(0)}\sigma^{(1)}\sigma^{(2)}$ is Coxeter--Knuth equivalent to $\tau^{(0)} \tau^{(1)} \sigma^{(0)} x \hat{\sigma}^{(1)}\sigma^{(2)}$. Let $\ell_j = \ell(\sigma^{(j)})$ for $j=0,1,2$. Since $\tau$ is increasing and $\tau^{(j)}_i > \sigma^{(j)}_i$ for $j=0,1$, we can apply Lemma~\ref{lem:slide} $\ell_0+\ell_1$ times to move letters of $\sigma^{(0)}\sigma^{(1)}$ left so that $\tau\sigma$ is Coxeter--Knuth equivalent to
  \begin{equation}
    \tau^{(0)}_1 \sigma^{(0)}_1 \cdots \tau^{(0)}_{\ell_0} \sigma^{(0)}_{\ell_0} x \tau^{(1)}_1 \sigma^{(1)}_1 \cdots \tau^{(1)}_{\ell_1} \sigma^{(1)}_{\ell_1} \sigma^{(2)} .
    \label{e:up}
  \end{equation}
  If $\ell_1=0$, then we may apply Lemma~\ref{lem:slide} to $\sigma^{(0)}_i$ for $i<\ell_0$ to deduce Eq.~\eqref{e:up} is Coxeter--Knuth equivalent to
  \[ \tau^{(0)}_1 \cdots \tau^{(0)}_{\ell_0} \sigma^{(0)}_1 \cdots \sigma^{(0)}_{\ell_0} x \sigma^{(2)}, \]
  as desired. If $\ell_1>0$ and $\tau^{(1)}_1 > \sigma^{(1)}_1+1$, then by Proposition~\ref{prop:drop}(3), $x < \sigma^{(1)}_1 < \tau^{(j)}_1$, and so $\CK_2(x \tau^{(1)}_1 \sigma^{(1)}_1) = \tau^{(1)}_1 x \sigma^{(1)}_1$. Applying Lemma~\ref{lem:slide} $\ell_1$ times to the letters of $\sigma^{(1)}$, then to $x$, then $\ell_0$ times to the letters of $\sigma^{(0)}$ shows Eq.~\eqref{e:up} is Coxeter--Knuth equivalent to $\tau^{(0)} \tau^{(1)} \sigma^{(0)} x \sigma^{(1)}\sigma^{(2)}$, as desired.

  Finally, suppose $\ell_1>0$ and $\tau^{(1)}_{1} = \sigma^{(1)}_i+i$ for $i\leq b$ with $b>0$ maximal. By Proposition~\ref{prop:drop}(3), $\CK_2(x \tau^{(1)}_1 \sigma^{(1)}_1) = \tau^{(1)}_1 x (\sigma^{(1)}_1+1)$. Then notice $\tau^{(1)}_{i} = \sigma^{(1)}_i+1$ for $i\leq b$, since otherwise $(\sigma^{(1)}_{i-1}+1) \tau^{(1)}_{i} \sigma^{(1)}_{i}$ is Coxeter equivalent to $(\sigma^{(1)}_{i-1}+1) \sigma^{(1)}_{i} \tau^{(1)}_{i}$, which is not reduced. Therefore $\CK_2((\sigma^{(1)}_{i-1}+1) \tau^{(1)}_{i} \sigma^{(1)}_{i}) = \tau^{(1)}_{i} (\sigma^{(1)}_{i-1}+1) (\sigma^{(1)}_{i}+1)$. Putting this together, Eq.~\eqref{e:up} is Coxeter--Knuth equivalent to
  \[ \tau^{(0)}_1 \sigma^{(0)}_1 \cdots \tau^{(0)}_{\ell_0} \sigma^{(0)}_{\ell_0} \tau^{(1)}_1 x \tau^{(1)}_2 (\sigma^{(1)}_1+1) \cdots \tau^{(1)}_{b} (\sigma^{(1)}_{b-1}+1) (\sigma^{(1)}_{b}+1) \tau^{(1)}_{b+1} \sigma^{(1)}_{b+1} \cdots \tau^{(1)}_{\ell_1} \sigma^{(1)}_{\ell_1} \sigma^{(2)} . \]
  Notice that the letters of $\sigma^{(1)}$ have now been raised to letters of $\hat{\sigma}^{(1)}$. Finally, apply Lemma~\ref{lem:slide} $\ell_1$ times to the letters of $\hat{\sigma}^{(1)}$, then to $x$, then $\ell_0$ times to the letters of $\sigma^{(0)}$ to see this is Coxeter--Knuth equivalent to $\tau^{(0)} \tau^{(1)} \sigma^{(0)} x \hat{\sigma}^{(1)}\sigma^{(2)}$, as desired.

  For $k>1$, the above case shows $\tau\sigma$ is Coxeter--Knuth equivalent to 
  \[ \tau^{(0)} x_1 \tau^{(1)} \cdots x_{k-1} \tau^{(k-1)} \tau^{(k)} \sigma^{(0)} \sigma^{(1)}\cdots \sigma^{(k-1)} x_k \hat{\sigma}^{(k)}\sigma^{(k+1)}. \]
  However, notice that this new pair factors uniquely by combining $\tau^{(k-1)}$ and $\tau^{(k)}$, as well as $\sigma^{(k-1)}$, $x$, and $\sigma^{(k)}_1\cdots\sigma^{(k)}_{\ell_k-1}$, then also combining $\sigma^{(k)}_{\ell_k}$ and $\sigma^{(k+1)}$. Therefore there are fewer factors in the result, so by induction, we may drop the remaining $x_1,\ldots,x_{k-1}$ as well, completing the proof  
\end{proof}

Extending Definition~\ref{def:drop-two}, we define the drop of any reduced word based on the rows of its descent tableau.

\begin{definition}
  Let $\rho$ be a reduced word, and let $(\rho^{(k)} | \cdots | \rho^{(1)})$ be an increasing decomposition of $\rho$. Define $\drop_i(\rho)$ by replacing $\rho^{(i+1)}\rho^{(i)}$ with $\drop(\rho^{(i+1)}\rho^{(i)})$. 
  \label{def:drop_i}
\end{definition}

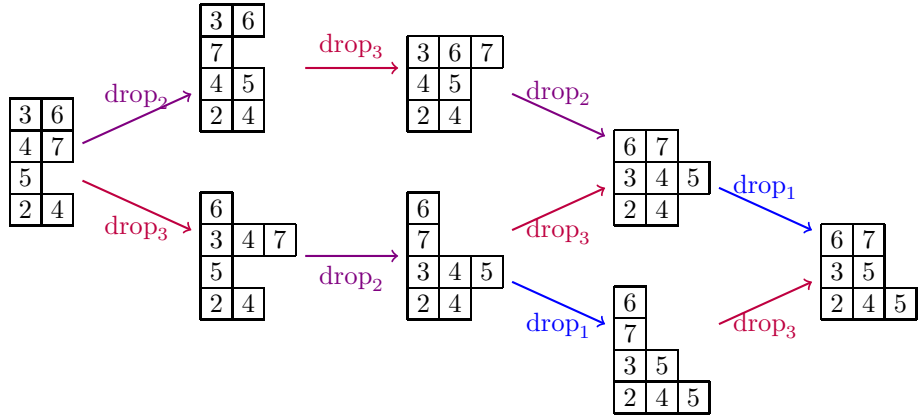
\begin{figure}[ht]  
  \begin{center}
    \begin{tikzpicture}[xscale=2.75,yscale=1.25]
      \node at (0,2)   (T02) {$\tableau{3 & 6 \\ 4 & 7 \\ 5 \\ 2 & 4}$}; 
      \node at (1,3)   (T13) {$\tableau{3 & 6 & \\ 7 \\ 4 & 5 \\ 2 & 4}$}; 
      \node at (1,1)   (T11) {$\tableau{6 \\ 3 & 4 & 7 \\ 5 \\ 2 & 4}$}; 
      \node at (2,3)   (T23) {$\tableau{\\ 3 & 6 & 7 \\ 4 & 5 \\ 2 & 4}$}; 
      \node at (2,1)   (T21) {$\tableau{6 \\ 7 \\ 3 & 4 & 5 \\ 2 & 4}$}; 
      \node at (3,2)   (T32) {$\tableau{\\ 6 & 7 \\ 3 & 4 & 5 \\ 2 & 4}$}; 
      \node at (3,0)   (T30) {$\tableau{6 \\ 7 \\ 3 & 5 \\ 2 & 4 & 5}$}; 
      \node at (4,1)   (T41) {$\tableau{\\ 6 & 7 \\ 3 & 5 \\ 2 & 4 & 5}$}; 
      \draw[thick,color=violet,->] (T02) -- (T13) node[midway,above] {$\drop_2$} ;
      \draw[thick,color=purple,->] (T02) -- (T11) node[midway,below] {$\drop_3$} ;
      \draw[thick,color=purple,->] (T13) -- (T23) node[midway,above] {$\drop_3$} ;
      \draw[thick,color=violet,->] (T11) -- (T21) node[midway,below] {$\drop_2$} ;
      \draw[thick,color=violet,->] (T23) -- (T32) node[midway,above] {$\drop_2$} ;
      \draw[thick,color=purple,->] (T21) -- (T32) node[midway,below] {$\drop_3$} ;
      \draw[thick,color=blue  ,->] (T21) -- (T30) node[midway,below] {$\drop_1$} ;
      \draw[thick,color=blue  ,->] (T32) -- (T41) node[midway,above] {$\drop_1$} ;
      \draw[thick,color=purple,->] (T30) -- (T41) node[midway,below] {$\drop_3$} ;
    \end{tikzpicture}
  \end{center}
  \caption{\label{fig:drop_braid}All drop sequences taking $\D(3,6,4,7,5,2,4)$ to the unique increasing tableau in its Coxeter--Knuth equivalence class.}
\end{figure}

We visualize increasing factorizations as tableaux with strictly increasing rows, and then the drop maps can be visualized as dropping cells in the tableaux, as shown in Fig.~\ref{fig:drop_braid}.

\begin{proposition}
  For $\rho$ a reduced word with run decomposition $(\rho^{(k)} | \cdots | \rho^{(1)})$, $\D(\rho)$ is an increasing Young tableau if and only if $\drop_i(\rho)=\rho$ for all $1\leq i<k$.
  \label{prop:drop-inc}
\end{proposition}

\begin{proof}
  If $\D(\rho)$ is increasing, then in Definition~\ref{def:drop-align} we will always left justify $\rho^{(j)}$ with respect to $\rho^{(j+1)}$ resulting in no $x_i$'s in the unique factorization of Proposition~\ref{prop:drop}, and so $\drop_j(\rho) = \rho$ for all $j$. Moreover, this is the only case in which there are no $x_i$'s, and so the only case when $\drop_j(\rho) = \rho$ for all $j$. 
\end{proof}

We apply the maps $\drop_i$ until reaching this terminal state, in which case we have an increasing Young tableau. To see that this is independent of the choice of which rows to drop when, we observe that these maps satisfy the nil-Hecke relations. 

\begin{theorem}
  The maps $\drop_i$ are well-defined on reduced words satisfying 
  \renewcommand{\theenumi}{\roman{enumi}}
  \begin{enumerate}
  \item $\drop_i \circ \drop_i = \drop_i$; \label{i:idem}
  \item $\drop_i \circ \drop_j = \drop_j \circ \drop_i$ for $|i-j|>1$; \label{i:commute}
  \item $\drop_i \circ \drop_{i+1} \circ \drop_i  = \drop_{i+1} \circ \drop_i \circ \drop_{i+1} $. \label{i:braid}
  \end{enumerate}
  \label{thm:braid}
\end{theorem}

\begin{proof}
  By the Drop Lemma, $\rho$ is Coxeter--Knuth equivalent to $\drop_i(\rho)$. In particular, when $\rho$ is reduced, so is $\drop_i(\rho)$, so we may iterate the maps.

  For relation~\eqref{i:idem}, notice that since $x_i < \sigma^{(i)}_1$ for $i=1,\ldots,k$, with notation as in Proposition~\ref{prop:drop}, the columns of $\D(\drop(\tau\sigma))$ are strict when the two rows are left justified. Therefore, by uniqueness of the factorization in Proposition~\ref{prop:drop}, there will be no $x_i$'s for $\drop(\tau\sigma)$. In particular, $\drop$ will act trivially.

  Relation~\eqref{i:commute} follows from the fact that $\drop_k$ considers only rows $k$ and $k+1$, so for $|i-j|>1$, the sets of indices $\{i,i+1\}$ and $\{j,j+1\}$ are disjoint. 
    
  Finally, for relation~\eqref{i:braid} it is enough to consider a three term factorization, say $\tau\sigma\rho$ with each of $\tau, \sigma, \rho$ increasing and the concatentation reduced. We must show that $\drop_1 \drop_2 \drop_1(\tau\sigma\rho)$ is equal to $\drop_2 \drop_1 \drop_2(\tau\sigma\rho)$. Align $\rho$ below $\sigma$ with letters of unsupported letters of $\sigma$ denoted $x_1,\ldots,x_k$, and then, maintaining that alignment, align $\sigma$ below $\tau$ with letters of unsupported letters of $\tau$ denoted $y_1,\ldots,y_{l}$; call this the \emph{initial alignment}. For a generic example, see Fig.~\ref{fig:triple_drop}. We consider blocks of the initial alignment with respect to the factorization of $\tau$ that put each $y_i$ at the end of a block as indicated.

  \begin{figure}[ht]
  \begin{displaymath}
    \arraycolsep=1.5pt
    \begin{array}{rcllll|lllll|l|ll|lll}
      \tau   & : & \times & \times & \times & y_1 & \times & \times & \times & \times & y_2 & y_3 & \times & y_4 & \times & & \\
      \sigma & : & \times & \times & x_1    &     & x_2    & \times & x_3     & \times &     &     & \times & & \times & & \\
      \rho   & : & \times & \times &        &     &        & \times &        &  \times &     &     & \times & & \times & \times & \times 
    \end{array} 
  \end{displaymath}
  \caption{\label{fig:triple_drop}An example of the drop alignment of $\rho$ below $\sigma$ below $\tau$.}
\end{figure}

  Consider first $\drop_1 \drop_2 \drop_1(\tau\sigma\rho)$. Applying $\drop_1(\tau\sigma\rho) = \tau\drop(\sigma\rho)$ results in all $x_i$'s moving down to row $1$. When subsequently applying $\drop_2$ to the result, in addition to the $y_i$'s dropping to row $2$, for every $x_i$, letting $j_i$ be smallest such that $y_{j_i}$ is right of $x_i$, some entry in row $3$ weakly right of $x_i$ and strictly left of $y_{j_i}$ will also drop to row $2$. The final application of $\drop_1$ now pulls down one cell from row $2$ for each $y_i$, and this will either be $y_i$ or some entry originally from $\tau$ that is weakly right of $y_i$. In particular, for $\drop_1 \drop_2 \drop_1(\tau\sigma\rho)$, we have
  \begin{itemize}
  \item the top row has no $y_i$'s and loses additional cells based on the drop alignment of $\sigma$ with $x_i$'s removed below $\tau$;
  \item the bottom row has all $x_i$'s and additional cells originally from $\tau$ based on the drop alignment of $\rho$ with $x_i$'s added below $\tau$.
  \end{itemize}

  Consider next $\drop_2 \drop_1 \drop_2(\tau\sigma\rho)$. Applying $\drop_2(\tau\sigma\rho) = \drop(\tau\sigma)\rho$ results in all $y_i$'s moving down to row $2$. When subsequently applying $\drop_1$ to the result, every $x_i$ will move down and, in addition, one cell originally from $\tau$ for every $y_i$ and this cell must be weakly right of $y_i$. The final application of $\drop_2$ will pull down one cell from $\tau$ for every $x_i$ for which did not sit directly below some $y_{j_i}$. In particular, the description above holds for the top and bottom rows of $\drop_2 \drop_1 \drop_2(\tau\sigma\rho)$, and so the two must correspond for all three rows.
\end{proof}

\begin{definition}
  For a reduced word $\rho$, the \emph{drop of $\rho$}, denoted by $\drop(\rho)$, is defined as follows: Let $\rho^{(k)}$ denote the $k$th row of $\D(\rho)$. Choose any $k$ such that either $\ell(\rho^{(k+1)}) > \ell(\rho^{(k)})$ or $\rho^{(k+1)}_j \leq \rho^{(k)}_j$ for some $j \leq \ell(\rho^{(k)})$. Replace $\rho^{(k+1)}\rho^{(k)}$ with $\drop(\rho^{(k+1)}\rho^{(k)})$. Repeat until the result is an increasing Young tableau.
  \label{def:drop}
\end{definition}

By Theorem~\ref{thm:braid}, the definition of $\drop(\rho)$ is independent of the order in which rows are consolidated, and so we have the following.

\begin{theorem}
  For a reduced word $\rho$, $\drop(\rho)$ is the unique Coxeter--Knuth equivalent reduced word for which the descent tableau is increasing and of partition shape. 
  \label{thm:drop}
\end{theorem}

\begin{proof}
  By Theorem~\ref{thm:braid}, the end result is independent of the order in which the $\drop_k$ are applied, and so $\drop(\rho)$ is well-defined. By Lemma~\ref{lem:drop}, each application of $\drop_k$ stays within the Coxeter--Knuth class, and so the theorem follows.
\end{proof}

Theorem~\ref{thm:drop} shows that each Coxeter--Knuth equivalence class has a unique representative $\pi$ characterized by the property $\D(\pi)$ is an increasing Young tableau of partition shape. In particular, combining this with Theorem~\ref{thm:deg-red}, this gives a new and simplified proof of the following equivalent formulation of Theorem~\ref{thm:EG}.

\begin{corollary}
  For $w$ a permutation, we have
  \begin{equation}
    \stanley_{w} = \sum_{\substack{\rho\in\R(w) \\ \D(\rho) \ \text{increasing}}} s_{\Des(\rho)} .
    \label{e:stanley-new}
  \end{equation}
    \label{cor:drop}
\end{corollary}

\subsection{Yamanouchi key tableaux}
\label{sec:positive-yamanouchi}

In order to give an explicit, direct formula for the Demazure expansion of a Schubert polynomial, we begin by characterizing the analogs of increasing Young tableaux that will give our canonical representatives for each Coxeter--Knuth equivalence class in the polynomial setting.

\begin{lemma}
  Given a weak composition $a$, there exists a unique $T_a\in\SKT(a)$ for which $\des(T_a) = a$. Moreover, for $T\in\SKT(a)$ non-virtual, we have $\des(T) \geq a$, where we define $b \geq a$ if and only if $b_1+\cdots+b_k \geq a_1+\cdots+a_k$ for all $k$.
  \label{lem:yam}
\end{lemma}

\begin{proof}
  We construct $T_a$ by filling the key diagram for $a$ with entries $n, n-1,\ldots,2,1$ left to right along rows beginning with the top. Then rows decrease left to right and columns decrease top to bottom, showing $T_a\in\SKT(a)$. Definition~\ref{def:des-key} will have $\hat{t}_i = t_i$ for all $i$, so $T_a$ will indeed have $\des(T_a)=a$. Uniqueness follows since any alternative filling $T \in\SKT(a)$ necessarily has $\flatten(\des(T)) \neq \flatten(\des(T_a)) = \flatten(a)$. Finally, the latter condition follows from the upper unitriangularity of Demazure characters with respect to monomials.
\end{proof}

Thus it makes sense to consider the set of reduced words that subordinate their Coxeter--Knuth equivalence classes.

\begin{definition}
  A reduced word $\rho$ is \emph{Yamanouchi} if for any non-virtual Coxeter-Knuth equivalent reduced word $\sigma$, we have $\des(\sigma) \geq \des(\rho)$. Denote the set of Yamanouchi reduced words for $w$ by $\Yam(w)$.
  \label{def:yam}
\end{definition}

For example, there are two Yamanouchi reduced words for the permutation $153264$, namely $(5,3,2,3,4)$ and $(3,5,2,3,4)$. Notice as well that we may reformulate the Demazure expansion for our running example as 
\begin{displaymath}
  \schubert_{153264} = \key_{(0,3,1,0,1)} + \key_{(0,3,2,0,0)} = \key_{\des(5,3,2,3,4)} + \key_{\des(3,5,2,3,4)}.
\end{displaymath}
Moreover, this expansion holds in general, giving the following formula.

\begin{theorem}
  Given a permutation $w$, we have 
  \begin{equation}
    \schubert_w = \sum_{\rho \in \Yam(w)} \key_{\des(\rho)} .
  \end{equation}
  \label{thm:yam}
\end{theorem}

\begin{proof}
  By \cite{Ass-W}(Theorem~3.33), the fundamental slide generating polynomial of a Coxeter--Knuth equivalence class is a single Demazure character. Therefore by Lemma~\ref{lem:yam}, each Coxeter--Knuth equivalence class has a unique element whose weak descent composition is dominated by every other element of the class, and so every Coxeter--Knuth equivalence class contains a unique Yamanouchi reduced word. By Lemma~\ref{lem:yam} again, the weak descent composition of the Yamanouchi reduced word indexes the Demazure character corresponding to the class. 
\end{proof}

The formula in Theorem~\ref{thm:yam} is still indirect since the definition of Yamanouchi requires consideration of the entire Coxeter--Knuth equivalence class. In order to avoid searching entire classes to find the Yamanouchi reduced words, we present an algorithm by which they can be constructed from the increasing Young tableaux with reduced reading words. Beginning with an increasing Young tableau, we raise letters from lower rows while staying within the same Coxeter--Knuth class by inverting the $\drop$ map from Definition~\ref{def:drop}. To begin, we must align.

\begin{definition}
  Given two increasing words $\sigma,\tau$ of lengths $s,t$, respectively, define the \emph{lift alignment of $\tau$ above $\sigma$} as follows: If $\tau_{t-i+1} \geq \sigma_{s-i+1}$ for all $i\leq\min(s,t)$, then right justify $\tau$ with respect to $\sigma$. Otherwise, set $i_1$ to be the minimum index such that $\tau_{t-i+1} < \sigma_{s-i+1}$, align $\tau_{t-i_1+2}\cdots\tau_{t}$ directly above $\sigma_{s-i+2}\cdots\sigma_{s}$ and iterate the process with the lift alignment of $\tau_{1}\cdots\tau_{t-i_1+1}$ above $\sigma_{1}\cdots\sigma_{s-i_1}$.
  \label{def:lift-align}
\end{definition}

Visually, right justify $\tau$ above $\sigma$ and, from right to left, for each not weakly strict column, slide entries of $\tau$ from that column onward left by one position. For example, Fig.~\ref{fig:drop-ex-2} shows the lift alignments for two pairs of increasing words.

Parallel to the drop case, if there are $k$ instances in the lift alignment of $\tau$ above $\sigma$ where a cell of $\sigma$ has no cell below it, then we denote these cells of $\sigma$ as $x_1,\ldots,x_k$ and factor $\sigma=\sigma^{(1)}x_1\sigma^{(2)}\cdots x_k\sigma^{(k+1)}$ and, correspondingly, $\tau=\tau^{(0)}\tau^{(1)}\cdots\tau^{(k+1)}$ as shown in Fig.~\ref{fig:lift_factor}. When the concatentation $\tau\sigma$ is a reduced word, then this factorization has the following properties.

\begin{figure}[ht]
  \begin{displaymath}
    \arraycolsep=1.5pt
    \begin{array}{lllllll}
      \tau^{(0)} & \tau^{(1)} & & \tau^{(2)} & \cdots & & \tau^{(k+1)}  \\
      & \sigma^{(1)} & x_1 & \sigma^{(2)} & \cdots & x_k & \sigma^{(k+1)} 
    \end{array} 
  \end{displaymath}
  \caption{\label{fig:lift_factor}An illustration of the lift alignment of $\tau$ above $\sigma$ and corresponding factorizations of $\sigma$ and $\tau$.}
\end{figure}

The following analog of Proposition~\ref{prop:drop} has a completely analogous proof.

\begin{proposition}
  Given two increasing words $\tau,\sigma$ such that $\tau\sigma$ is reduced, there is a unique factorization $\tau=\tau^{(0)}\tau^{(1)}\cdots \tau^{(k+1)}$ and $\sigma=\sigma^{(1)}x_1\cdots\sigma^{(k)}x_k\sigma^{(k+1)}$, with some $\tau^{(i)}$ or $\sigma^{(i)}$ possibly empty, such that
  \begin{enumerate}
  \item $\ell(\tau^{(j)}) = \ell(\sigma^{(j)})$ for $j=1,\ldots,k+1$;
  \item $\tau^{(j)}_{\ell(\tau^{(j)})} < x_j$ for $j=1,\ldots,k$; 
  \item $\tau^{(j)}_i \geq \sigma^{(j)}_i$ for $j=1,\ldots,k+1$ with equality only if $j>1$ and $\tau^{(j)}_{1} = x_{j-1}-1$.
  \end{enumerate}  
  \label{prop:lift}
\end{proposition}

We call this factorization the \emph{lift} alignment since we will lift the unblocked cells $x_1,\ldots,x_k$ from $\sigma$ up to $\tau$ without changing the Coxeter--Knuth equivalence class.

\begin{definition}
  Given two increasing words $\tau,\sigma$ such that $\tau\sigma$ is reduced, define 
  \[ \lift(\tau\sigma) = \tau^{(0)} \tau^{(1)} x_1 \tau^{(2)} \cdots x_k \tau^{(k+1)} \sigma^{(1)} \check{\sigma}^{(2)} \cdots \check{\sigma}^{(k+1)}, \]
  where the factorization is the unique one in Proposition~\ref{prop:lift} and for $1\leq j \leq k$
  \[ \check{\sigma}^{(j+1)}_i = \left\{ \begin{array}{ll}
    \sigma^{(j+1)}_i-1 & \text{for} \ 1 \leq i \leq b_j \\
    \sigma^{(j+1)}_i & \text{for} \ b_j+1 \leq i \leq \ell(\sigma^{(j+1)})
  \end{array} \right. \]
  for $b_j = \max\{b \mid  \tau^{(j+1)}_i = \sigma^{(j+1)}_i = x_j+i \ \forall 1 \leq i \leq b \}$ if $x_j = \sigma^{(j+1)}_1-1$ or else $0$.
  \label{def:lift-two}
\end{definition}

For example, lifting the aligned words in Fig.~\ref{fig:drop-ex-2} results in the words in Fig.~\ref{fig:drop-ex}.

\begin{lemma}[Lift Lemma]
  Given two increasing words $\tau,\sigma$ such that $\tau\sigma$ is reduced, $\lift(\tau\sigma)$ is Coxeter--Knuth equivalent to $\tau\sigma$.
  \label{lem:lift}
\end{lemma}

\begin{proof}
  With the factorization as denoted in Proposition~\ref{prop:lift}, consider first the case $k=1$, in which we must show $\tau^{(0)} \tau^{(1)} \tau^{(2)} \sigma^{(1)} x \sigma^{(2)}$ is Coxeter--Knuth equivalent to $\tau^{(0)} \tau^{(1)} x \tau^{(2)} \check{\sigma}^{(1)} \sigma^{(2)}$. Let $\ell_j = \ell(\tau^{(j)})$ for $j=0,1,2$. Since $\tau$ is increasing and $\tau^{(1)}_i > \sigma^{(1)}_i$, we can apply Lemma~\ref{lem:slide} $\ell_1$ times to move letters of $\sigma^{(1)}$ left so that $\tau\sigma$ is Coxeter--Knuth equivalent to
  \begin{equation}
    \tau^{(0)} \tau^{(1)}_1 \sigma^{(1)}_1 \cdots \tau^{(1)}_{\ell_1} \sigma^{(1)}_{\ell_1} \tau^{(2)} x \sigma^{(2)} .
    \label{e:up2}
  \end{equation}
  
  If $\ell_2=0$, then since $\sigma^{(1)}_i<x$ for all $i$, we may apply Lemma~\ref{lem:slide} to $\sigma^{(1)}_i$ for $i \leq \ell_1$ to deduce Eq.~\eqref{e:up2} is Coxeter--Knuth equivalent to $\tau^{(0)}\tau^{(1)}x\sigma^{(1)}$ as desired.

  If $\ell_2>0$ and $\tau^{(2)}_1 > \sigma^{(2)}_1$, then by Proposition~\ref{prop:lift}(3), $x < \sigma^{(2)}_1 < \tau^{(2)}_1$, and so $\CK_2(\tau^{(2)}_1 x \sigma^{(2)}_1) = x \tau^{(2)}_1 \sigma^{(2)}_1$. Applying Lemma~\ref{lem:slide} $\ell_2$ times to $x$ and the letters of $\sigma^{(2)}$, then $\ell_1$ times to the letters of $\sigma^{(1)}$ shows Eq.~\eqref{e:up} is Coxeter--Knuth equivalent to $\tau^{(0)} \tau^{(1)} x \tau^{(2)} \sigma^{(1)}\sigma^{(2)}$, as desired.

  Finally, suppose $\ell_2>0$ and $\tau^{(2)}_{i} = \sigma^{(2)}_i = x+i$ for $i\leq b$ with $b>0$ maximal. By Proposition~\ref{prop:lift}(3) and the increasing property of $\tau$, we have $\sigma^{(2)}_j <= \tau^{(2)}_j < \tau^{(2)}_{j+1}$. Thus applying Lemma~\ref{lem:slide} shows Eq.~\eqref{e:up2} is Coxeter--Knuth equivalent to
  \begin{equation}
    \tau^{(0)} \tau^{(1)}_1 \sigma^{(1)}_1 \cdots \tau^{(1)}_{\ell_1} \sigma^{(1)}_{\ell_1} \tau^{(2)}_1 x \tau^{(2)}_2 \sigma^{(2)}_1 \cdots \tau^{(2)}_{\ell_2} \sigma^{(2)}_{\ell_2-1} \sigma^{(2)}_{\ell_2} .    
    \label{e:up3}
  \end{equation}
  Based on the assumptions on $b$, we have the elementary relations
  \[ \CK_2(\tau^{(2)}_{i} \sigma^{(2)}_{i-1} \sigma^{(2)}_{i}) = \left\{ \begin{array}{ll}
    \sigma^{(2)}_{i-1} \tau^{(2)}_{i} \sigma^{(2)}_{i} & \text{if} \ i>b, \\
    \sigma^{(2)}_{i-1} \tau^{(2)}_{i} (\sigma^{(2)}_{i}-1) & \text{if} \ i \leq b,
  \end{array} \right. \]
  and $\CK_2(\tau^{(2)}_1 x \sigma^{(2)}_1) = x \tau^{(2)}_1 (\sigma^{(2)}_1-1)$. Thus Eq.~\eqref{e:up3} is Coxeter--Knuth equivalent to
  \begin{displaymath}
    \tau^{(0)} \tau^{(1)}_1 \sigma^{(1)}_1 \cdots \tau^{(1)}_{\ell_1} \sigma^{(1)}_{\ell_1} x \tau^{(2)}_1 (\sigma^{(2)}_1-1) \cdots \tau^{(2)}_b (\sigma^{(2)}_b-1) \tau^{(2)}_{b+1} \sigma^{(2)}_{b+1} \cdots \tau^{(2)}_{\ell_2} \sigma^{(2)}_{\ell_2} .    
  \end{displaymath}
  Notice that the letters of $\sigma^{(2)}$ have now been decremented to letters of $\check{\sigma}^{(2)}$. Finally, apply Lemma~\ref{lem:slide} $\ell_2$ times to the letters of $\check{\sigma}^{(2)}$, then $\ell_1$ times to the letters of $\sigma^{(1)}$ to see this is Coxeter--Knuth equivalent to $\tau^{(0)} \tau^{(1)} x \tau^{(2)} \sigma^{(1)} x \check{\sigma}^{(2)}$, as desired.

  For $k>1$, the above case shows $\tau\sigma$ is Coxeter--Knuth equivalent to 
  \[ \tau^{(0)} \tau^{(1)} x_1 \tau^{(2)} \cdots \tau^{(k)} \sigma^{(1)} \check{\sigma}^{(2)} x_2 \sigma^{(3)} \cdots x_k \sigma^{(k+1)}. \]
  However, notice that this new pair factors uniquely by combining $\tau^{(0)}$ and $\tau^{(1)}_1$, combining $\tau^{(1)}_2\cdots\tau^{(1)}_{\ell_1}$, $x$, and $\tau^{(2)}$, and combining $\sigma^{(1)}$ and $\check{\sigma}^{(2)}$. Therefore there are fewer factors in the result, so by induction on $k$, we may drop the remaining $x_2,\ldots,x_{k}$ as well, completing the proof.
\end{proof}

Extending Definition~\ref{def:lift-two}, we define $\lift_i$ for a reduced word as follows.

\begin{definition}
  Let $\rho$ be a reduced word, and let $(\rho^{(k)} | \cdots | \rho^{(1)})$ be an increasing decomposition of $\rho$. Define $\lift_i(\rho)$ by replacing $\rho^{(i+1)}\rho^{(i)}$ with $\lift(\rho^{(i+1)}\rho^{(i)})$. 
  \label{def:lift_i}
\end{definition}

\begin{figure}[ht]
  \begin{center}
    \begin{tikzpicture}[xscale=3.3,yscale=2.5]
      \node at (0,1) (a1) {$\tableau{6 & 9 \\ 3 & 7 & 8 \\ 2 & 3 & 5 & 9 \\ 1 & 2 & 4 & 5 & 6}$};
      \node at (1,0) (b0) {$\tableau{6 & 9 \\ 3 & 7 & 8 \\ 2 & 3 & 4 & 5 & 9 \\ 1 & 2 & 4 & 6}$};
      \node at (1,1) (b1) {$\tableau{6 & 9 \\ 3 & 7 & 8 & 9 \\ 2 & 3 & 5 \\ 1 & 2 & 4 & 5 & 6}$};
      \node at (1,2) (b2) {$\tableau{6 & 7 & 9 \\ 3 & 8 \\ 2 & 3 & 5 & 9 \\ 1 & 2 & 4 & 5 & 6}$};
      \node at (2,0) (c0) {$\tableau{6 & 9 \\ 3 & 7 & 8 & 9 \\ 2 & 3 & 4 & 5 & 6 \\ 1 & 2 & 4}$};
      \node at (2,1) (c1) {$\tableau{6 & 7 & 9 \\ 3 & 8 \\ 2 & 3 & 4 & 5 & 9 \\ 1 & 2 & 4 & 6}$};
      \node at (2,2) (c2) {$\tableau{6 & 7 & 8 & 9 \\ 3 & 8 \\ 2 & 3 & 5 \\ 1 & 2 & 4 & 5 & 6}$};
      \node at (3,1) (d1) {$\tableau{6 & 7 & 8 & 9 \\ 3 & 8 \\ 2 & 3 & 4 & 5 & 6 \\ 1 & 2 & 4}$};
      \draw[thick,color=blue  ,->] (a1) -- (b0) node[midway,right] {$\lift_1$} ;
      \draw[thick,color=purple,->] (a1) -- (b1) node[midway,above] {$\lift_2$} ;
      \draw[thick,color=violet,->] (a1) -- (b2) node[midway,left ] {$\lift_3$} ;
      \draw[thick,color=blue  ,->] (b1) -- (c0) node[midway,right] {$\lift_1$} ;
      \draw[thick,color=blue  ,->] (b2) -- (c1) node[midway,right] {$\lift_1$} ;
      \draw[thick,color=violet,->] (b0) -- (c1) node[midway,left ] {$\lift_3$} ;
      \draw[thick,color=violet,->] (b1) -- (c2) node[midway,left ] {$\lift_3$} ;
      \draw[thick,color=blue  ,->] (c2) -- (d1) node[midway,right] {$\lift_1$} ;
      \draw[thick,color=violet,->] (c0) -- (d1) node[midway,left ] {$\lift_3$} ;
    \end{tikzpicture}
  \end{center}
  \caption{\label{fig:key-raise}The lifting algorithm applied to an increasing Young tableau.}
\end{figure}
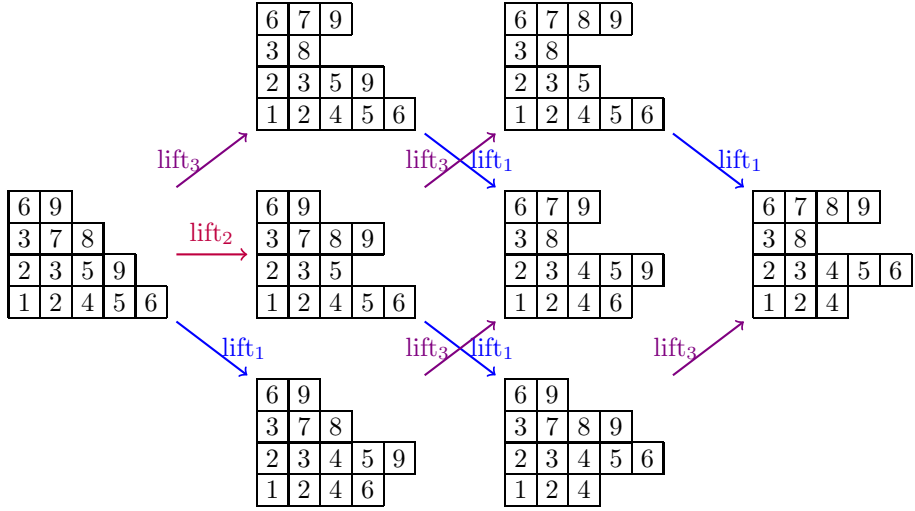

Reversing the symmetric situation, we can define an \emph{increasing key tableau} with the corresponding row condition and such that the result is lift-invariant.

\begin{definition}
  A key tableau $T$ is \emph{increasing} if the rows are strictly increasing (left to right) and $\lift_i(T)=T$ for all $i$.
  \label{def:inc_key}
\end{definition}

Unlike the case for $\drop_i$, we do not wish to apply the maps $\lift_i$ indiscriminantly until reaching some increasing key tableau. Foremost among the reasons is that these maps do not, in general, satisfy the nil-Hecke relations. For example, Fig.~\ref{fig:key-raise} shows all nontrivial $\lift_i$ operators applied the leftmost tableau $P$, which is the unique increasing Young tableau in the Coxeter--Knuth class. Notice
\[ \lift_1 \circ \lift_2 \circ \lift_1 (P) = \lift_1 (P) \neq \lift_1 \circ \lift_2 (P) = \lift_2 \circ \lift_1 \circ \lift_2 (P). \]
Furthermore, both the third and fourth tableaux (from the left) in the middle row of Fig.~\ref{fig:key-raise} are increasing, though only the fourth is Yamanouchi. 

Nevertheless, we do have a \emph{canonical lifting path} from the unique increasing Young tableau to the unique Yamanouchi tableau. To define this path, we say that $\lift_i$ acts \emph{faithfully} on a tableau $T$ if $\lift_i(T)\neq T$, and we extend this notion to a sequence of lifting maps in the obvious way.

\begin{lemma}
  Let $P$ be an increasing Young tableau with reduced reading word, and suppose each operator in $\lift_{i_k} \circ \cdots \circ \lift_{i_1}(P)$ acts faithfully. Then 
  \[ \Des( \lift_{i_k} \circ \cdots \circ \lift_{i_1} (P) ) = s_{i_k} \cdots s_{i_1} \cdot \Des(P), \]
  where $s_i$ acts on compositions by interchanging parts in positions $i$ and $i+1$.
  \label{lem:extremal}
\end{lemma}

\begin{proof}
  In the notation of Definition~\ref{def:lift-align}, since $\ell(\tau^{(i)}) = \ell(\sigma^{(i)})$ for all $i>0$, raising the $x_i$'s from $\sigma$ to $\tau$ in $\lift(\tau\sigma)$ precisely exchanges the two parts of $\Des(\tau\sigma)$ provided $\tau^{(0)} = \varnothing$. Considering $P$, since columns strictly increase bottom to top, for any entry $z$ in row $i$, there are at least as many $x<z$ in row $i$ as there are $y<z$ in row $j>i$. In particular, taking $j=i+1$, ensuring that $\tau^{(0)}$ must be empty when applying $\lift_i$. Furthermore, when $\tau^{(0)}=\varnothing$, entries that move upward maintain their columns, and only larger entries move left. Therefore this property is maintained throughout lifting, ensuring $\tau^{(0)}=\varnothing$ at each step.
\end{proof}

For $i\leq j$, define the \emph{lifting sequence} $\lift_{[i,j]}$ by
\begin{equation}
  \lift_{[i,j]} = \lift_j \circ \lift_{j-1} \circ \cdots \circ \lift_i .
  \label{e:lift}
\end{equation}
We say that a lifting sequence $\lift_{[i,j]}$ acts faithfully on a tableau $T$ if $\lift_i$ acts faithfully on $T$ and $\lift_k$ acts faithfully on $\lift_{[i,k-1]}(T)$ for all $i < k \leq j$.

\begin{definition}
  For $P$ an increasing Young tableau whose row reading word is reduced, define the \emph{lift of $P$}, denoted by $\lift(P)$, to be the tableau of key shape constructed as follows. Set $T_0 = P$, and for $k>0$,
  \begin{enumerate}
  \item if $\lift_i(T_{k-1})=T_{k-1}$ for all $i$, then $\lift(P)=T_{k-1}$;
  \item otherwise, set $T_k = \lift_{[i_k,j_k]} (T_{k-1})$ where
    \begin{enumerate}
    \item $j_k$ is the maximum $j$ for which there exists $i \leq j$ such that $\lift_{[i,j]}$ acts faithfully on $T_{k-1}$, and
    \item $i_k$ is the minimum $i \leq j_k$ for which $\lift_{[i,j_k]}$ acts faithfully on $T_{k-1}$.
    \end{enumerate}
  \end{enumerate}
  \label{def:lift}
\end{definition}

For example, with $P$ the leftmost tableau in Fig.~\ref{fig:key-raise}, we have 
\[ \lift(P) = \lift_{[1,1]} \circ \lift_{[2,3]} (P) = \lift_1 \circ \lift_3 \circ \lift_2 (P) \]
which is the rightmost tableau in Fig.~\ref{fig:key-raise} and is Yamanouchi.

\begin{theorem}
  For $\rho$ a reduced word with $\D(\rho)$ an increasing Young tableau, the word $\lift(\rho)$ is Yamanouchi.
  \label{thm:lift}
\end{theorem}

\begin{proof}
  By Theorem~\ref{thm:wdeg-red}, there is a $\des$-preserving isomorphism, say $\theta$, from the Coxeter--Knuth equivalence class of $\rho$ to $\SKT(a)$ for some weak composition $a$. Moreover, by Theorem~\ref{thm:drop} and Proposition~\ref{prop:key-stable}, we must have $\mathrm{sort}(a) = \Des(\rho)$.

  Given any $T\in\SKT(a)$ for which $\des(T)$ rearranges the parts of $a$, if $S\in\SKT(a)$ has $\des(S)=\des(T)$, then $S=T$, which is to say that for each weak composition $b$ that rearranges $a$, there is at most one element of $\SKT(a)$ with weak descent composition $b$. In particular, $\theta(\rho)$ is determined as is $\theta(\sigma)$ for any $\sigma$ obtained from $\rho$ by a sequence of lifts by Lemma~\ref{lem:extremal}. Moreover, by Lemma~\ref{lem:extremal}, we may extend the maps $\lift_i$ to those $\SKT(a)$ whose weak descent compositions rearrange $a$ so that they intertwine with the isomorphism $\theta$. Thus it suffices to show that Definition~\ref{def:lift} applied via $\theta$ to $\SKT(a)$ gives the Yamanouchi standard key tableau whose row reading word is the reverse of the identity.
  
  Let $P = \theta(\rho)\in\SKT(a)$, and set $\lambda=\Des(P)$ say with length $\ell$. Let $\alpha$ be the composition obtained by removing zero parts of $a$. After the first pass of Definition~\ref{def:lift} applied to $P$, we have $T_1 = \lift_{[i_1,j_1]}(P)\in\SKT(a)$ where $j_1$ is the largest index $j$ for which $\lambda_j \neq \alpha_j$ but $\lambda_k = \alpha_k$ for all $k>j$, and $i_1$ is the largest index $i<j_1$ for which $\lambda_i = \alpha_{j_1}$. Therefore $\Des(T_1) = s_{j_1} \cdots s_{i_1} \cdot \lambda$ agrees with $\alpha$ in all positions $j \geq j_1$. Furthermore, reading the rows of $T_1$ left to right from the top down to $j_1$, we precisely have the reverse of the identity. Thus we may proceed by induction on the first $j_1-1 < \ell$ rows of $T_1$. 
\end{proof}

For an example of the maps $\lift_i$ induced on elements of $\SKT(a)$ whose weak descent compositions rearrange $a$, see Fig.~\ref{fig:SKT-raise}. Compare this with Fig.~\ref{fig:key-raise}. 

\begin{figure}[ht]
  \begin{center}
    \begin{tikzpicture}[xscale=3.3,yscale=2.5]
      \node at (0,1) (a1) {$\tableau{14 & 13 & 10 & 6 \\ 12 & 11 \\ 9 & 8 & 7 & 2 & 1 \\ 5 & 4 & 3}$};
      \node at (1,0) (b0) {$\tableau{14 & 13 & 10 & 1 \\ 12 & 11 \\ 9 & 8 & 7 & 6 & 5 \\ 4 & 3 & 2}$};
      \node at (1,1) (b1) {$\tableau{14 & 13 & 10 & 9 \\ 12 & 11 \\ 8 & 7 & 6 & 2 & 1 \\ 5 & 4 & 3}$};
      \node at (1,2) (b2) {$\tableau{14 & 13 & 12 & 6 \\ 11 & 10 \\ 9 & 8 & 7 & 2 & 1 \\ 5 & 4 & 3}$};
      \node at (2,0) (c0) {$\tableau{14 & 13 & 10 & 9 \\ 12 & 11 \\ 8 & 7 & 6 & 5 & 4 \\ 3 & 2 & 1}$};
      \node at (2,1) (c1) {$\tableau{14 & 13 & 12 & 1 \\ 11 & 10 \\ 9 & 8 & 7 & 6 & 5 \\ 4 & 3 & 2}$};
      \node at (2,2) (c2) {$\tableau{14 & 13 & 12 & 11 \\ 10 & 9 \\ 8 & 7 & 6 & 2 & 1 \\ 5 & 4 & 3}$};
      \node at (3,1) (d1) {$\tableau{14 & 13 & 12 & 11 \\ 10 & 9 \\ 8 & 7 & 6 & 5 & 4 \\ 3 & 2 & 1}$};
      \draw[thick,color=blue  ,->] (a1) -- (b0) node[midway,right] {$\lift_1$} ;
      \draw[thick,color=purple,->] (a1) -- (b1) node[midway,above] {$\lift_2$} ;
      \draw[thick,color=violet,->] (a1) -- (b2) node[midway,left ] {$\lift_3$} ;
      \draw[thick,color=blue  ,->] (b1) -- (c0) node[midway,right] {$\lift_1$} ;
      \draw[thick,color=blue  ,->] (b2) -- (c1) node[midway,right] {$\lift_1$} ;
      \draw[thick,color=violet,->] (b0) -- (c1) node[midway,left ] {$\lift_3$} ;
      \draw[thick,color=violet,->] (b1) -- (c2) node[midway,left ] {$\lift_3$} ;
      \draw[thick,color=blue  ,->] (c2) -- (d1) node[midway,right] {$\lift_1$} ;
      \draw[thick,color=violet,->] (c0) -- (d1) node[midway,left ] {$\lift_3$} ;
    \end{tikzpicture}
  \end{center}
  \caption{\label{fig:SKT-raise}The lifting algorithm applied to extremal standard key tableaux.}
\end{figure}
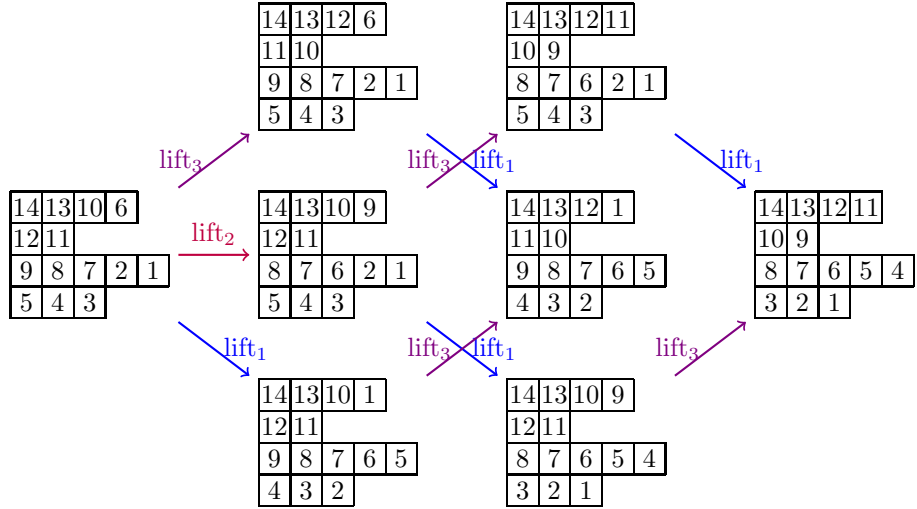

In particular, on the level of generating polynomials, we have the following improvement of Theorem~\ref{thm:LS} parallel to Theorem~\ref{thm:EG} and Corollary~\ref{cor:drop}.

\begin{corollary}
  For $w$ a permutation, we have
  \begin{equation}
    \schubert_{w} = \sum_{\substack{\rho\in\R(w) \\ \D(\rho) \ \text{increasing}}} \key_{\des(\lift(\rho))} .
    \label{e:schubert-new}
  \end{equation}
  \label{cor:schubert-new}
\end{corollary}

For example, lifting the increasing Young tableaux in Figure~\ref{fig:increasing} we arrive at the Yamanouchi key tableaux in Figure~\ref{fig:yam}, as promised.

%
\section{Insertion algorithms}
%
\label{sec:insertion}

Edelman and Greene \cite{EG87} define an insertion algorithm mapping reduced words to pairs of Young tableaux where the left is increasing and the right is standard. In this context, the left tableau gives the canonical Coxeter--Knuth equivalence class representative for obtaining the Schur expansion of a Stanley symmetric function, and the right tableau gives an explicit bijection between elements of the Coxeter--Knuth equivalence class and standard Young tableaux of fixed shape. We recall their definitions and main results for the purpose of generalizing them to the polynomial setting. In the generalization, the left tableau will be a Yamanouchi key tableau, and the right tableau will be a standard key tableau. Thus the left tableau gives the canonical Coxeter--Knuth equivalence class representative for obtaining the Demazure expansion of a Schubert polynomial, and the right tableau gives an explicit bijection between elements of the Coxeter--Knuth equivalence class and standard key tableaux of fixed shape.

\subsection{Edelman--Greene insertion}
\label{sec:insertion-EG}

Edelman and Greene \cite{EG87}(Definition~6.20) defined the following procedure for inserting a letter into an increasing tableau. 

\begin{definition}[\cite{EG87}]
  Let $P$ be an increasing Young tableau, and let $x$ be a positive integer. Let $P_i$ be the $i$th lowest row of $P$. Define the \emph{Edelman--Greene insertion of $x$ into $P$}, denoted by $P \stackrel{x}{\leftarrow}$, as follows. Set $x_0=x$ and for $i\geq 0$, insert $x_i$ into $P_{i+1}$ as follows: if $x_i \geq z$ for all $z\in P_{i+1}$, place $x_i$ at the end of $P_{i+1}$ and stop; otherwise, let $x_{i+1}$ denote the smallest element of $P_{i+1}$ such that $x_{i+1}>x_i$ (we say that $x_i$ \emph{bumps} $x_{i+1}$ in row $i+1$), replace $x_{i+1}$ by $x_i$ in $P_{i+1}$ only if $x_{i+1} \neq x_i+1$ or $x_i$ is not already in $P_{i+1}$, and continue.
  \label{def:insert-EG}
\end{definition}

This algorithm generalizes the insertion algorithm of Schensted \cite{Sch61}, building on work of Robinson \cite{Rob38}, later generalized by Knuth \cite{Knu70}. Robinson--Schensted insertion becomes a bijective correspondence between permutations and pairs of standard Young tableaux by constructing a second tableau to track the order in which new cells are added. The pair is typically denoted by $(P,Q)$, where $P$ is called the \emph{insertion tableau}, and $Q$ is called the \emph{recording tableau}.

Similarly, we construct the Edelman--Greene correspondence of a reduced word $\rho = (\rho_k,\ldots,\rho_1)$ by successively inserting the letters of $\rho$ from $k$ to $1$ into the empty tableau to create the \emph{Edelman--Greene insertion tableau} of $\rho$, denoted by $P(\rho)$. For example, Figure~\ref{fig:EG-P} shows the Edelman-Greene insertion tableau for the reduced word $\rho = (3,6,4,7,5,2,4)$.

\begin{figure}[ht]
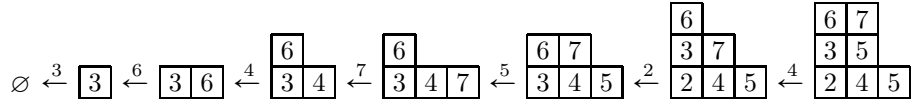

  \begin{displaymath}
    \arraycolsep=2pt
    \begin{array}{ccccccccccccccc}
    \raisebox{-1.8\cellsize}{$\varnothing$}
    & \raisebox{-1.75\cellsize}{$\xleftarrow{3}$}
    & \tableau{\\ \\ 3} 
    & \raisebox{-1.75\cellsize}{$\xleftarrow{6}$}
    & \tableau{\\ \\ 3 & 6} 
    & \raisebox{-1.75\cellsize}{$\xleftarrow{4}$}
    & \tableau{\\ 6 \\ 3 & 4} 
    & \raisebox{-1.75\cellsize}{$\xleftarrow{7}$}
    & \tableau{\\ 6 \\ 3 & 4 & 7} 
    & \raisebox{-1.75\cellsize}{$\xleftarrow{5}$}
    & \tableau{\\ 6 & 7 \\ 3 & 4 & 5} 
    & \raisebox{-1.75\cellsize}{$\xleftarrow{2}$}
    & \tableau{6 \\ 3 & 7 \\ 2 & 4 & 5}
    & \raisebox{-1.75\cellsize}{$\xleftarrow{4}$}
    & \tableau{6 & 7 \\ 3 & 5 \\ 2 & 4 & 5} 
    \end{array}    
  \end{displaymath}
  \caption{\label{fig:EG-P}The Edelman-Greene insertion tableau for the reduced word $\rho = (3,6,4,7,5,2,4)$.}
\end{figure}

Since Edelman--Greene insertion adds a single cell to an existing Young diagram, when inserting $\rho_i$, we create the \emph{Edelman--Greene recording tableau} of $\rho$, denoted by $Q(\rho)$, by adding a cell with entry $k-i+1$ into the position of $P(\rho_k,\ldots,\rho_i) \setminus P(\rho_k,\ldots,\rho_{i+1})$. For example, Figure~\ref{fig:EG-Q} shows the Edelman-Greene recording tableau for the reduced word $\rho = (3,6,4,7,5,2,4)$.

\begin{figure}[ht]
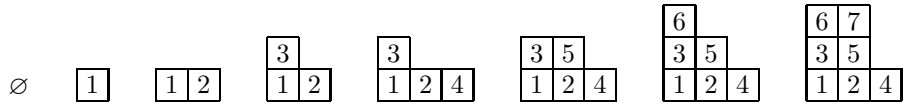

  \begin{displaymath}
    \arraycolsep=9pt
    \begin{array}{cccccccc}
    \raisebox{-1.8\cellsize}{$\varnothing$}
    & \tableau{\\ \\ 1} 
    & \tableau{\\ \\ 1 & 2} 
    & \tableau{\\ 3 \\ 1 & 2} 
    & \tableau{\\ 3 \\ 1 & 2 & 4}  
    & \tableau{\\ 3 & 5 \\ 1 & 2 & 4}  
    & \tableau{6 \\ 3 & 5 \\ 1 & 2 & 4}  
    & \tableau{6 & 7 \\ 3 & 5 \\ 1 & 2 & 4}  
    \end{array}    
  \end{displaymath}
  \caption{\label{fig:EG-Q}The Edelman-Greene recording tableau for the reduced word $\rho = (3,6,4,7,5,2,4)$.}
\end{figure}

Edelman and Greene derived many properties of this generalized insertion algorithm, including that the insertion tableau $P(\rho)$ defined by inserting $\rho_k,\ldots,\rho_1$ into the empty tableau is a well-defined increasing tableau whose row reading word is a reduced word for $w$. They proved the following \cite{EG87}(Theorem~6.24) relating their insertion to Coxeter--Knuth equivalence.

\begin{theorem}[\cite{EG87}] 
  For reduced words $\sigma,\tau$, we have $P(\sigma)=P(\tau)$ if and only if $\sigma$ and $\tau$ are Coxeter-Knuth equivalent.
  \label{thm:cox-knu}
\end{theorem}

Theorem~\ref{thm:cox-knu} gives canonical Coxeter--Knuth equivalence class representatives as the reduced words occuring as reading words of increasing Young tableaux. Thus the result follows as well from Theorem~\ref{thm:drop}.

Further, Edelman and Greene characterize how the recording tableaux differ for two reduced words that differ by an elementary Coxeter--Knuth equivalence. Refining \cite{EG87}(Definition~6.14), we have the following definition from \cite{Ass15}. 

\begin{definition}
  The \emph{elementary dual equivalence involutions}, denoted by $d_i$, act on standard Young tableaux by
  \begin{equation}
    d_i(T) = \left\{ \begin{array}{rl}
      s_{i-1} \cdot T & \text{if $i+1$ lies between $i$ and $i-1$ in $\word(T)$, } \\
      s_{i} \cdot T & \text{if $i-1$ lies between $i$ and $i+1$ in $\word(T)$, } \\
      T & \text{if $i$ lies between $i-1$ and $i+1$ in $\word(T)$, }
    \end{array} \right.
    \label{e:ede}
  \end{equation}
  where $s_i$ acts by interchanging $i$ and $i+1$, and $\word(T)$ is the row reading word.
\end{definition}

For example, the elementary dual equivalence involutions on standard Young tableaux of shape $(3,2)$ are shown in Fig.~\ref{fig:deg}.

\begin{figure}[ht]  
  \begin{center}
    \begin{tikzpicture}[xscale=2.5,yscale=1]
      \node at (0,0)   (a) {$\tableau{3 & 4 \\ 1 & 2 & 5}$}; 
      \node at (1,0)   (b) {$\tableau{2 & 4 \\ 1 & 3 & 5}$}; 
      \node at (2,0)   (c) {$\tableau{2 & 5 \\ 1 & 3 & 4}$}; 
      \node at (3,0)   (d) {$\tableau{3 & 5 \\ 1 & 2 & 4}$}; 
      \node at (4,0)   (e) {$\tableau{4 & 5 \\ 1 & 2 & 3}$}; 
      \draw[thick,color=blue  ,<->] (a.008) -- (b.172) node[midway,above] {$d_2$} ;
      \draw[thick,color=violet,<->] (a.352) -- (b.188) node[midway,below] {$d_3$} ;
      \draw[thick,color=purple,<->] (b) -- (c) node[midway,above] {$d_4$} ;
      \draw[thick,color=blue  ,<->] (c) -- (d) node[midway,above] {$d_2$} ;
      \draw[thick,color=violet,<->] (d.008) -- (e.172) node[midway,above] {$d_3$} ;
      \draw[thick,color=purple,<->] (d.352) -- (e.188) node[midway,below] {$d_4$} ;
    \end{tikzpicture}
  \end{center}
  \caption{\label{fig:deg}The elementary dual equivalence involutions on $\SYT(3,2)$.}
\end{figure}
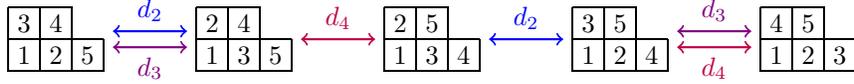

Edelman and Greene \cite{EG87}(Corollary~6.15) relate Coxeter--Knuth equivalence with dual equivalence through the recording tableaux as follows.

\begin{theorem}[\cite{EG87}] 
  For reduced words $\sigma,\tau$, we have $Q(\sigma)=d_i(Q(\tau))$ if and only if $\sigma = \CK_{n-i+1}(\tau)$.
  \label{thm:dual-CK}
\end{theorem}

Theorem~\ref{thm:dual-CK} follows from Theorem~\ref{thm:deg-red}, proving that Edelman--Greene insertion establishes a $\Des$-preserving bijection between elements of a Coxeter--Knuth equivalence class and standard Young tableaux of fixed shape.

Edelman and Greene use their insertion and recording tableaux to establish the following bijective correspondence \cite{EG87}(Theorem~6.25).

\begin{corollary}[\cite{EG87}]
  The Edelman--Greene correspondence $\rho \rightarrow \left(P(\rho), Q(\rho)\right)$ establishes a bijection
  \begin{equation}
    \R(w) \stackrel{\sim}{\longrightarrow} \bigsqcup_{\lambda} \left(\mathrm{Inc}_{\lambda}(w) \times \SYT(\lambda) \right),
\end{equation}
  where $\mathrm{Inc}_{\lambda}(w)$ is the set of increasing reduced words $\sigma$ for $w$ such that $\Des(\sigma)=\lambda$. Moreover, under this correspondence, $\Des(\rho) = \Des(Q(\rho))$. 
  \label{cor:biject-EG}
\end{corollary}

Taking fundamental quasisymmetric generating functions gives Theorem~\ref{thm:EG}.

\subsection{Weak insertion}
\label{sec:insertion-weak}

We generalize Edelman-Greene insertion to an algorithm on reduced words that outputs a pair tableau of \emph{key shape} such that the insertion tableau is a Yamanouchi key tableau (in particular, it is increasing) and the recording tableau is a standard key tableau. Leveraging Definition~\ref{def:insert-EG} along with Definitions~\ref{def:drop} and \ref{def:lift}, we have the following.

\begin{definition}
  For $P$ a Yamanouchi key tableau and $x$ a positive integer, define the \emph{weak insertion of $x$ into $P$}, denoted by $P \stackrel{x}{\weakarrow}$, to be $\lift(\drop(P)\stackrel{x}{\leftarrow})$.
  \label{def:insert-weak}
\end{definition}

Construct the \emph{weak correspondence} of a reduced word $\rho = (\rho_k,\ldots,\rho_1)$ by successively inserting the letters of $\rho$ from $k$ to $1$ into the empty tableau to create the \emph{weak insertion tableau} of $\rho$, denoted by $\wP(\rho)$. For example, Figure~\ref{fig:weak-P} shows the weak insertion tableau for the reduced word $\rho = (3,6,4,7,5,2,4)$.

\begin{figure}[ht]
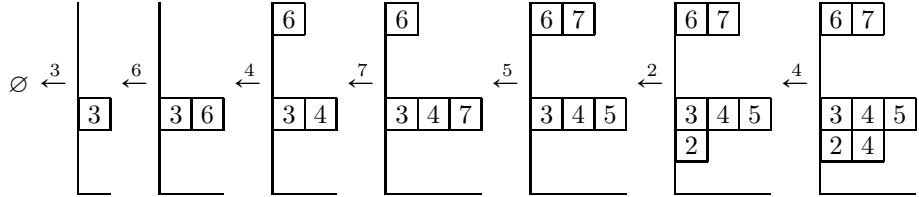

  \begin{displaymath}
    \arraycolsep=2pt
    \begin{array}{ccccccccccccccc}
    \raisebox{-1.8\cellsize}{$\varnothing$}
    & \raisebox{-1.75\cellsize}{$\xleftarrow{3}$} & \vline\tableau{\\ \\ \\ 3 \\ \\ \\\hline} 
    & \raisebox{-1.75\cellsize}{$\xleftarrow{6}$} & \vline\tableau{\\ \\ \\ 3 & 6 \\ \\ \\\hline} 
    & \raisebox{-1.75\cellsize}{$\xleftarrow{4}$} & \vline\tableau{6 \\ \\ \\ 3 & 4 \\ \\ \\\hline} 
    & \raisebox{-1.75\cellsize}{$\xleftarrow{7}$} & \vline\tableau{6 \\ \\ \\ 3 & 4 & 7 \\ \\ \\\hline} 
    & \raisebox{-1.75\cellsize}{$\xleftarrow{5}$} & \vline\tableau{6 & 7 \\ \\ \\ 3 & 4 & 5 \\ \\ \\\hline}
    & \raisebox{-1.75\cellsize}{$\xleftarrow{2}$} & \vline\tableau{6 & 7 \\ \\ \\ 3 & 4 & 5 \\ 2 \\ \\\hline}
    & \raisebox{-1.75\cellsize}{$\xleftarrow{4}$} & \vline\tableau{6 & 7 \\ \\ \\ 3 & 4 & 5 \\ 2 & 4 \\ \\\hline} 
    \end{array}    
  \end{displaymath}
  \caption{\label{fig:weak-P}The weak insertion tableau for the reduced word $\rho = (3,6,4,7,5,2,4)$.}
\end{figure}

Parallel to Theorem~\ref{thm:cox-knu}, we have the following.

\begin{theorem}
  For reduced words $\sigma,\tau$, we have $\wP(\sigma)=\wP(\tau)$ if and only if $\sigma$ and $\tau$ are Coxeter-Knuth equivalent.
  \label{thm:weak-knu}
\end{theorem}

\begin{proof}
  By Theorem~\ref{thm:drop}, dropping a word so that the result is an increasing Young tableau maintains the Coxeter--Knuth equivalence class. By Theorem~\ref{thm:lift}, the Yamanouchi words are constructed by lifting the increasing Young tableaux, and by Theorem~\ref{thm:yam} they are the canonical representatives for each weak dual equivalence class. By Theorem~\ref{thm:wdeg-red}, Coxeter--Knuth equivalence classes are weak dual equivalence classes, and so the result follows.
\end{proof}

In order to define a weak recording tableau, we must show that the successive shapes when insertion a word are nested. To this end, we have the following.

\begin{lemma}
  Let $\rho$ be a Yamanouchi reduced word and $x$ a letter such that the concatenation $\rho x$ is reduced, and set $\sigma = \lift(\drop(\rho x))$. Then $\des(\rho)_i \leq \des(\sigma)_i$ for all $i$, and if $j$ is the unique index such that $\des(\sigma)_j = \des(\rho)_j+1$, then $\des(\sigma)_j \neq \des(\rho)_i$ for all $i<j$.
  \label{lem:SKT}
\end{lemma}

\begin{proof}
  By Lemma~\ref{lem:extremal}, the locations of the nonempty rows of $\wP(\rho)$ are determined by the entries in the first column of $P(\rho)$, and these are the nonzero entries of $\des(\lift(\drop(\rho)))$. Thus inserting an additional letter will increase the length of one (possibly empty) row. Therefore $\des(\rho)_i \leq \des(\sigma)_i$ for all $i$.

  Let $j$ denote the unique index for which $\des(\sigma)_j = \des(\rho)_j+1$. Suppose, for contradiction, $\des(\sigma)_j=1$ and $\des(\rho)_i = 1$ for some $i<j$. Note the entries in $\wP(\sigma)$ in rows $i,j$ must be $i,j$, respectively. Furthermore, in $P(\sigma)$, both $i$ and $j$ are singleton cells, so the Edelman--Greene insertion of $x$ into $\rho$ must have caused $i$ to bump $j$. In particular, $j$ must have been a singleton cell of $P(\rho)$, but this contradicts that $\des(\sigma)_j = \des(\rho)_j+1$. Similarly, if $\des(\sigma)_j=\des(\rho)_i =c>1$ for some $i<j$, then the entry in row $j$ column $c$ must be strictly larger than the entry in row $i$ column $c$, and the same bumping argument applies, presenting the same contradiction. Thus we must have $\des(\sigma)_j \neq \des(\rho)_i $ for all $i<j$ as desired.
\end{proof}

We may now define the \emph{weak recording tableau} of $\rho$, denoted by $\wQ(\rho)$, by adding a cell with entry $i$ into the position of $\wP(\rho_k,\ldots,\rho_i) \setminus \wP(\rho_k,\ldots,\rho_{i+1})$. For example, Figure~\ref{fig:weak-Q} shows the weak recording tableau for $\rho = (3,6,4,7,5,2,4)$.

\begin{figure}[ht]
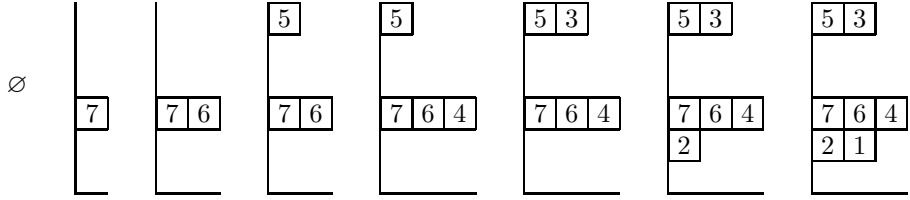

  \begin{displaymath}
    \arraycolsep=9pt
    \begin{array}{cccccccc}
    \raisebox{-1.8\cellsize}{$\varnothing$}
    & \vline\tableau{\\ \\ \\ 7 \\ \\ \\\hline} 
    & \vline\tableau{\\ \\ \\ 7 & 6 \\ \\ \\\hline} 
    & \vline\tableau{5 \\ \\ \\ 7 & 6 \\ \\ \\\hline} 
    & \vline\tableau{5 \\ \\ \\ 7 & 6 & 4 \\ \\ \\\hline} 
    & \vline\tableau{5 & 3 \\ \\ \\ 7 & 6 & 4 \\ \\ \\\hline} 
    & \vline\tableau{5 & 3 \\ \\ \\ 7 & 6 & 4 \\ 2 \\ \\\hline} 
    & \vline\tableau{5 & 3 \\ \\ \\ 7 & 6 & 4 \\ 2 & 1 \\ \\\hline} 
    \end{array}    
  \end{displaymath}
  \caption{\label{fig:weak-Q}The weak recording tableau for the reduced word $\rho = (3,6,4,7,5,2,4)$.}
\end{figure}

\begin{theorem}
  For $\rho$ a reduced word, the weak recording tableau $\wQ(\rho)$ is a standard key tableau.
  \label{thm:weak-Q}
\end{theorem}

\begin{proof}
  By Lemma~\ref{lem:SKT}, the successive shapes created during the weak insertion of $\rho$ are nested, making the recording tableau well-defined. Entries are added to th recording tableau in decreasing order, ensuring that rows decrease left to right. The latter condition of Lemma~\ref{lem:SKT} ensures that if $i$ is added above an entry $k$, then the length of the lower row containing $k$ must be longer than that length of the upper row containing $i$. In particular, there must be an entry $j$ right of $k$ added before $i$, thus $j>i$. Therefore $\wQ(\rho)$ is indeed a standard key tableau.
\end{proof}

We can characterize how the recording tableaux differ for two reduced words that differ by an elementary Coxeter--Knuth equivalence using elementary weak dual equivalences \cite{Ass-W}(Definition~3.21). 

\begin{definition}[\cite{Ass-W}]
  The \emph{elementary weak dual equivalence involutions}, denoted by $\wD_i$, act on standard key tableaux as follows. Let $u,v,w$ be the cells with entries $i-1,i,i+1$ taken in column reading order. Then
  \begin{equation}
    \wD_i (T) = \left\{ \begin{array}{rl}
      b_{i}  (T) & \mbox{if $u,w$ are in the same row and $v$ is not} , \\
      s_{i-1}(T) & \mbox{else if $v$ has entry $i+1$} , \\
      s_{i}  (T) & \mbox{else if $v$ has entry $i-1$} , \\
      T & \mbox{otherwise},
    \end{array} \right.
    \label{e:ewde}
  \end{equation}
  where $b_{j}$ cycles $j-1,j,j+1$ so that $j$ shares a row with $j \pm 1$.
\end{definition}

The elementary weak dual equivalence involutions on standard key tableaux of shape $(0,3,0,2)$ are shown in Fig.~\ref{fig:wdeg}.

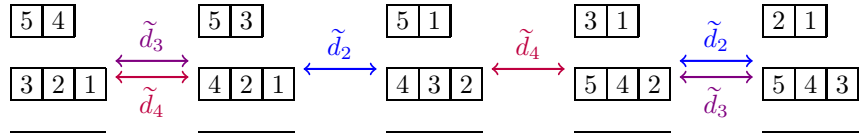
\begin{figure}[ht]  
  \begin{center}
    \begin{tikzpicture}[xscale=2.5,yscale=1]
      \node at (4,0)   (a) {$\vline\tableau{2 & 1 \\ \\ 5 & 4 & 3 \\ \\\hline}$}; 
      \node at (3,0)   (b) {$\vline\tableau{3 & 1 \\ \\ 5 & 4 & 2 \\ \\\hline}$}; 
      \node at (2,0)   (c) {$\vline\tableau{5 & 1 \\ \\ 4 & 3 & 2 \\ \\\hline}$}; 
      \node at (1,0)   (d) {$\vline\tableau{5 & 3 \\ \\ 4 & 2 & 1 \\ \\\hline}$}; 
      \node at (0,0)   (e) {$\vline\tableau{5 & 4 \\ \\ 3 & 2 & 1 \\ \\\hline}$}; 
      \draw[thick,color=blue  ,<->] (a.172) -- (b.008) node[midway,above] {$\wD_2$} ;
      \draw[thick,color=violet,<->] (a.188) -- (b.352) node[midway,below] {$\wD_3$} ;
      \draw[thick,color=purple,<->] (b) -- (c) node[midway,above] {$\wD_4$} ;
      \draw[thick,color=blue  ,<->] (c) -- (d) node[midway,above] {$\wD_2$} ;
      \draw[thick,color=violet,<->] (d.172) -- (e.008) node[midway,above] {$\wD_3$} ;
      \draw[thick,color=purple,<->] (d.188) -- (e.352) node[midway,below] {$\wD_4$} ;
    \end{tikzpicture}
  \end{center}
  \caption{\label{fig:wdeg}The elementary weak dual equivalence involutions on $\SKT(0,3,0,2)$.}
\end{figure}

We relate Coxeter--Knuth equivalence with weak dual equivalence through the weak recording tableaux as follows.

\begin{theorem}
  For reduced words $\sigma,\tau$, we have $\wQ(\sigma)=\wD_i(\wQ(\tau))$ if and only if $\sigma = \CK_{i}(\tau)$.
  \label{thm:wdual-CK}
\end{theorem}

\begin{proof}
  By \cite{Ass-W}(Theorem~3.24), the bijection $\Phi:\SKT(a)\rightarrow\SYT(\mathrm{sort}(a))$ that drops entries in a key tableau to partition shape, replaces $i$ with $n-i+1$, and sorts columns to increase bottom to top intertwines the elementary weak dual equivalence involutions on standard key tableaux with the elementary dual equivalence involutions on standard Young tableaux by $\Phi(\wD_i(T)) = d_{n-i+1}(\Phi(T))$. The result now follows from Theorem~\ref{thm:dual-CK}.
\end{proof}

The weak insertion and recording tableaux establish the following bijective correspondence, parallel to Corollary~\ref{cor:biject-EG}.

\begin{corollary}
  The weak correspondence $\rho \rightarrow \left(\wP(\rho), \wQ(\rho)\right)$ establishes a bijection
  \begin{equation}
    \R(w) \stackrel{\sim}{\longrightarrow} \bigsqcup_{a} \left( \Yam_a(w) \times \SKT(a) \right) ,
  \end{equation}
  where $\Yam_a(w)$ is the set of Yamanouchi reduced words $\sigma$ for $w$ such that $\des(\sigma)=a$. Moreover, under this correspondence, $\des(\rho) = \des(\wQ(\rho))$. 
  \label{cor:biject-wEG}
\end{corollary}

Taking fundamental slide generating polynomials gives Corollary~\ref{cor:schubert-new}.

%
%

\bibliographystyle{plain} 
\bibliography{weak_EG}

\end{document}